\documentclass[reqno]{amsart}
\usepackage{amsmath,amssymb,mathrsfs}
\usepackage{graphicx,float,cite,times}
\usepackage{graphicx,subfigure,caption}
\usepackage{bm}
\usepackage[all]{xy}
\usepackage{float}

\usepackage{comment}
\includecomment{note}
\usepackage[marginpar=1in]{geometry}
\usepackage[
colorlinks,
citecolor=magenta,
urlcolor=blue,
linkcolor=magenta,
allcolors=black,
backref=page
]{hyperref}

\theoremstyle{definition}%%%%%zhengti
\newtheorem{theo}{Theorem}[section]
\newtheorem{coro}[theo]{Corollary}
\newtheorem{lemm}[theo]{Lemma}
\newtheorem{prop}[theo]{Proposition}
\newtheorem{rema}[theo]{Remark}
\newtheorem{defi}[theo]{Definition}

\newtheorem{conj}{Conjecture}

\numberwithin{equation}{section}

\newcommand{\bC}{{\mathbb{C}}}

\newcommand{\bN}{{\mathbb{N}}}
\newcommand{\bR}{{\mathbb{R}}}
\newcommand{\bZ}{{\mathbb{Z}}}

\newcommand{\tc}{{\mathtt{c}}}

\newcommand{\ca}{{\mathfrak{a}}}

\newcommand{\cc}{{\mathfrak{c}}}

\newcommand{\cs}{{\mathfrak{s}}}
\newcommand{\cC}{{\mathfrak{C}}}
\newcommand{\cD}{{\mathfrak{D}}}
\newcommand{\cI}{{\mathfrak{I}}}
\newcommand{\cF}{{\mathfrak{F}}}

\newcommand{\cN}{{\mathfrak{N}}}

\allowdisplaybreaks

\begin{document}
%\large
\title[Nonlinear Schr\"odinger Equations With Quasi-Periodic Initial Data I.]{
%Combinatorial analysis for infinite-imensional ODEs
%with higher-dimensional alternating discrete convolution\\
%Part I: The
Existence, Uniqueness and Asymptotic Dynamics of Nonlinear Schr\"odinger Equations With Quasi-Periodic Initial Data:\\I. The standard NLS
}

\dedicatory{Dedicated to the memory of Thomas Kappeler}

\author{David Damanik}
\address{\scriptsize (D. Damanik)~Department of Mathematics, Rice University, 6100 S. Main Street, Houston, Texas
77005-1892}
\email{\color{magenta}damanik@rice.edu}
\thanks{The first author (D. Damanik) was supported by NSF grant DMS--2054752}

\author{Yong Li}
\address{\scriptsize  (Y. Li)~Institute of Mathematics, Jilin University, Changchun 130012, P.R. China. School of Mathematics and Statistics, Center for Mathematics and Interdisciplinary Sciences, Northeast Normal University, Changchun, Jilin 130024, P.R.China.}
\email{\color{magenta}liyong@jlu.edu.cn}
\thanks{The second author (Y. Li) was supported in part by National Basic Research Program of China (2013CB834100), and NSFC (12071175).}

\author{Fei Xu}
\address{\scriptsize (F. Xu)~Institute of Mathematics, Jilin University, Changchun 130012, P.R. China.}
\email{\color{magenta}stuxuf@outlook.com}
\thanks{The third author (F. Xu) is supported by the China post-doctoral grant (BX20240138). He is sincerely grateful for the invitation to give a remote talk at UCLA on January 23, 2024, where the existence and uniqueness results on the (derivative) NLS with quasi-periodic initial data were announced publicly.}

%\thanks{{\bf Data Availability} On behalf of all authors, the corresponding author declares that there is no data associate for the submission.}
%
%\thanks{{\bf Conflict of Interest}~On behalf of all authors, the corresponding author states that there is no conflict of
%interest.}

\begin{abstract}
This is the first part of a two-paper series studying nonlinear Schr\"odinger equations with quasi-periodic initial data. In this paper, we consider the standard nonlinear Schr\"odinger equation. Under the assumption that the Fourier coefficients of the initial data obey a power-law upper bound, we establish local existence of a solution that retains quasi-periodicity in space with a slightly weaker Fourier decay. Moreover, the solution is shown to be unique within this class of quasi-periodic functions. In addition, for the nonlinear Schr\"odinger equation with small nonlinearity, within the time scale, as the small parameter of nonlinearity tends to zero, we prove that the nonlinear solution converges asymptotically to the linear solution with respect to both the sup-norm $\|\cdot\|_{L_x^\infty(\mathbb R)}$ and the Sobolev-norm $\|\cdot\|_{H^s_x(\mathbb R)}$.

The proof proceeds via a consideration of an associated infinite system of coupled ordinary differential equations for the Fourier coefficients and a combinatorial analysis of the resulting tree expansion of the coefficients.
For this purpose, we introduce a Feynman diagram for the Picard iteration and $\ast^{[\cdot]}$ to denote the complex conjugate label.
\end{abstract}

\maketitle
\tableofcontents

\section{Introduction}\label{intro}

This is the first part of a two-paper series dedicated to studying nonlinear Schr\"odinger-type equations with quasi-periodic initial data, given by the following Fourier series,
\begin{align}\label{id0}
u(0,x)&=V(x)\nonumber\\
&=\sum_{n\in\mathbb Z^\nu}c(n)e^{{\rm i}\langle n\rangle x}, \quad x\in\mathbb R.
\end{align}
Here, $2\leq\nu<\infty$ is the dimension of a given frequency vector $\omega=(\omega_j)_{1\leq j\leq\nu}\in\mathbb R^\nu$, $\langle n\rangle=n\cdot\omega=\sum_{j=1}^{\nu}n_j\omega_j$ stands for the inner product w.r.t. the given frequency vector $\omega$, and $\{c(n)\}_{n\in\mathbb Z^\nu}$ represents the collection of known initial Fourier coefficients.

Throughout, the frequency vector $\omega$ is assumed to be non-resonant, that is, rationally independent, which means that $\langle n\rangle=0$ implies $n=0\in\mathbb Z^\nu$.

This two-paper series includes the following two parts:

I. The Standard Nonlinear Schr\"odinger Equation
\begin{align}\label{cnls}
\tag{cNLS}{\rm i}\partial_tu+\partial_{xx}u\pm|u|^{2p}u=0,\quad 1\leq p\in\mathbb N.
\end{align}

II. The Derivative Nonlinear Schr\"odinger Equation
\begin{align}\label{dnls}
\tag{dNLS}{\rm i}\partial_tu+\partial_{xx}u-{\rm i}\partial_x(|u|^{2}u)=0.
\end{align}

%\subsection{Quasi-periodic Cauchy problem}
%For the sake of convenience, we introduce some universal notations and symbols on PDEs and their solutions here which will be used throughout this paper.
%
%Let $\partial_t$ and $\partial_x$ be the derivative w.r.t. time and space respectively. Let $u=u(t,x)$ be a solution to a general PDE
%\begin{align}\label{pde}
%\tag{PDE}\mathscr F(t,x,u,\partial_tu,\partial_xu,\partial_x^2u,\cdots)=0.
%\end{align}
%
%{\bf Quasi-periodic initial data (initial data as a quasi-periodic function)} has the following spatially quasi-periodic Fourier series
%\begin{align}\label{id}
%\tag{QPID}u(0,x)=\sum_{n\in\bZ^\nu}c(n)e^{{\rm i}\langle n\rangle x}.
%\end{align}

For the sake of convenience, we refer to \eqref{cnls}/\eqref{dnls} together with \eqref{id0} as the associated quasi-periodic Cauchy problem. What we are interested in is to study the (global) existence and uniqueness of spatially quasi-periodic solutions with the same frequency vector as the initial data to these equations. Such a solution is defined by the following spatially quasi-periodic Fourier series
%\marginpar{It may be a good idea to rename the display (2.17), for example (QPID).}
\begin{align}\label{ses}
u(t,x)=\sum_{n\in\bZ^\nu}c(t,n) e^{{\rm i}\langle n\rangle x},
\end{align}
where $c(0,n)=c(n)$ for all $n\in\bZ^\nu$, and $c(t,n)$ is a family of unknown Fourier coefficients for all $n\in\bZ^\nu$ on a suitable time interval.

This two-paper series has the following three sources of motivation.

%The first one is a communication with the late Thomas Kappeler. In 2021, after we submitted our manuscript on the generalized KdV equation with quasi-periodic initial data \cite{DLX2021}, we received positive feedback from Thomas. He expressed that he liked our work very much and asked that if our approach can be applied to other equations such as nonlinear Schr\"odinger equation with higher order non-linearity? Since then, we have discussed it several times by email. However, unfortunately, he passed away on May 30, 2022.  Though we continue to work on this plan and try our best to better understand it to dedicate the memory of Thomas, a great mathematician and a nice friend.

The first one is a communication with the late Thomas Kappeler. In 2021, after we submitted our manuscript on the generalized KdV equation with quasi-periodic initial data \cite{DLX2021}, we received positive feedback from him. He asked whether our approach could be applied to other equations such as the nonlinear Schr\"odinger equation with higher order non-linearity. Since then, we have discussed it several times by email. However, he unfortunately passed away on May 30, 2022. With respect and admiration we dedicate this two-paper series to the memory of Thomas Kappeler, a great mathematician.

The second one is from Klaus's work \cite{K2023}. He pointed out that it is still an open problem to understand the existence, uniqueness and large time behaviour of the solutions to the nonlinear Schr\"odinger equation with initial data that do not tend to zero as $|x|\rightarrow\infty$, even for the completely integrable, one-dimensional, cubic nonlinear Schr\"odinger equation
\begin{align}
{\rm i}\partial_tu+\partial_x^2u-2|u|^2u=0.
\end{align}

The third one is related to the the so-called Deift conjecture for the celebrated KdV equation
\begin{align}\label{kdv}
\tag{KdV}\partial_tu+\partial_x^3u+u\partial_xu=0.
\end{align}
Inspired by some global results on the periodic initial data problem for \eqref{kdv}, see \cite{L1975CPAM,MT1976CPAM}, and the discovery of quasi-crystals, Percy Deift proposed the following:
\begin{conj}(\hspace{-0.1mm}\cite{deift2008,deift2017})
If the initial data is almost periodic, the solution of \eqref{kdv} evolves almost periodically in time. That is, the solution $u(t,x)$ to \eqref{kdv} with almost periodic initial data is almost periodic in time $t$ and retains the same spatial almost periodicity as the initial data for all times.
\end{conj}
Regarding this conjecture, there are some past and recent results in different categories: the almost periodicity and quasi-periodicity of initial data, positive and negative answers.

In the $1970$s, Lax \cite{L1975CPAM} and McKean and Trubowitz \cite{MT1976CPAM} considered \eqref{kdv} with periodic initial data and they obtained a global result: the solution is almost periodic in time and it retains the same spatial periodicity as the initial data for all times. This is also one of the primary sources of motivation for the Deift conjecture.

In the quasi-periodic setting, Tsugawa obtained a local result under a polynomial decay condition in the Fourier space, i.e., $|c(n)|\lesssim (1+|n|)^{-\rho}$, where $\rho$ is large. The proof method he used is based on Bourgain's Fourier restriction norm method.

Later, also in the quasi-periodic setting, under an exponential decay condition in the Fourier space, i.e., $|c(n)|\lesssim e^{-\kappa|n|}$, Damanik and Goldstein proved a global result. They first used an explicit combinatorial analysis method to obtain a local result, and then applied the complete integrability structure of \eqref{kdv}, that is, the Lax pair structure, to get a global result \cite{DG2017JAMS}.
%Generally speaking, the key ingredient from a local result to a global one is to obtain a uniform length of time
%extension.
In general, the crucial step in extending from a local to a global result is to achieve a uniform time extension.
How did they succeed in doing that? First, with the Lax pair at hand, \eqref{kdv} can be viewed as an iso-spectral evolution of the associated dynamically defined Schr\"odinger operator
\[H_t=-\partial_x^2-\frac{u}{6},\]
where $u$ is the solution to \eqref{kdv}. They then applied a bi-correspondence between the exponential decay of the spectral gaps of the Schr\"odinger operator $H_t$
and the exponential decay of the family of the Fourier coefficients of the solution to \eqref{kdv}; see \cite{DG2014PMIHES}.

%That is, as we all known, \eqref{kdv} enjoys the Lax pair structure $(L,B)$, where
%
%and it can be viewed as an iso-spectral evolution of the Lax equation
%\[\frac{{\rm d}}{{\rm d}t}L=[L,B]\]
%$B$ is certain skew-symmetric operator.
%,
%
%Damanik and Goldstein

In the general setting, that is, with initial data being almost periodic in space, the Deift conjecture has been partially solved under certain assumptions on the Schr\"odinger operator $H_0$, i.e., a homogeneous spectrum along with the so-called reflectionlessness property. Under these assumptions, the resulting class of potential/operators can then be associated with a torus of dimension given by the number of the gaps of the spectrum in two different ways: one can either associate (i)~Dirichlet data or (ii)~pass to the dual group of the fundamental group of the complement of the spectrum. Then the KdV flow can be related to (Dubrovin/linear) flows on either of these tori, and one can establish existence of solutions there and then pull them back. On the first torus, Binder et al obtained a global result \cite{BDGL2018DUKE} if the spectrum of $H_0$ has only absolutely continuous part and it is ``thick" enough (i.e., it satisfies the so-call Craig-type condition plus homogeneity). On the second torus, there are existence \cite{EVY2019TAMS} and uniqueness \cite{LY2020JFA} results for higher-order KdV flows.

%Above results are all negative answers to the Deift conjecture.
The aforementioned results provide positive answers to the Deift conjecture. There are two recent works regarding a negative answer. In \cite{D21}, Damanik, Lukic, Volberg, and Yuditskii put forward their belief that this conjecture is not true in general and described a program to construct a counterexample to it, that is, an almost periodic function whose evolution under the KdV equation is not almost periodic in time. Later, in \cite{CKV},  Chapouto, Killip, and Visan constructed a counterexample to disprove this conjecture by choosing a rationally independent combination of square waves as initial data for \eqref{kdv}, that is,
\[V(x)=\text{sq}(\alpha_1x)+\text{sq}(\alpha_2x),\]
where~sq is a $2\pi$-periodic square wave, sq$(x)=$~sgn($\sin x$), and $\alpha_1, \alpha_2$~are rationally independent. Then they used a nonlinear smoothing effect to prove that the solution is not almost periodic in space at some time.

Though the Deift conjecture has been disproved, in the spirit of this conjecture, it remains an interesting and open problem to study almost periodicity in both time and space for PDEs such as KdV, NLS and other interesting equations.

We also wish to mention \cite{B1993, DSS2020, KST2017, O2015, O2015SIAMMA} as a partial list of papers containing related work.

%Motivated by Thomas's question and the Deift conjecture, we build this plan, especially for the most attractive nonlinear Schr\"odinger equations.
Inspired by Kappeler's question and the Deift conjecture, we devote this two-paper series to a study of nonlinear Schr\"odinger equations with quasi-periodic initial data.

\medskip

{{\bf Hypothesis.}}\quad Throughout this series, we will assume some decay
%\footnote{Throughout this paper, by the decay of a function, we refer the decay of its Fourier coefficients.}
condition on the initial data in the Fourier space, that is, on the Fourier coefficients w.r.t. the modulus $|n|$ of the dual variable $n\in\bZ^\nu$, which will be used in both \eqref{cnls} and \eqref{dnls}.  In fact, we will use two different kinds of decay conditions. The first one is the so-called polynomial decay condition: we say that a spatially quasi-periodic function $V$ is {\bf $\mathtt r$-polynomially decaying} for some positive constant $\mathtt r>0$ if its Fourier coefficients satisfy
\begin{align}
|\hat f(n)|\lesssim(1+|n|)^{-\mathtt r},\quad n\in\bZ^\nu.
\end{align}
The second one is the so-called exponential decay condition, that is, if the above polynomial decay condition is replaced by the following exponential decay condition
\begin{align}
|\hat f(n)|\lesssim e^{-\kappa|n|},\quad n\in\bZ^\nu,
\end{align}
we say that such a function is {\bf $\kappa$-exponentially decaying}. Here we call $\mathtt r$ or $\kappa$ {\bf the decay rate} in both cases.

\medskip

{\bf Main Results.}\quad Roughly speaking, under a polynomial/exponential decay condition in the Fourier space, we can obtain a local result for \eqref{cnls}/\eqref{dnls} respectively. Furthermore, under a certain suitable condition, we can prove that the former can be up to any assigned time horizon. Also, for both \eqref{cnls} and \eqref{dnls} in a weak nonlinear setting, within the time scale, the nonlinear solution is asymptotic to the linear one with respect to both the sup-norm $\|\cdot\|_{L_x^\infty(\mathbb R)}$ and the Sobolev-norm $\|\cdot\|_{H^s_x(\mathbb R)}$. See Theorem \autoref{cnlsth} in this paper and Theorem 1.1 in the second paper in this series, \cite{DLX24II}, for the detailed statements of our main results.

\begin{rema}

%\marginpar{{\color{magenta}Yong: Will these statements cause any negative effects (especially for the next paper) if Fei's talk information and Papenburg's preprint are added here?} {\color{blue}It is possible, but not saying anything could also have negative effects. The referee could evaluate the work based on the assumption that Papenburg was first and this paper was released several months later.}}
(a)~The existence and uniqueness results were publicly announced in January 2024 in a talk given by Fei Xu in the Analysis and PDE Seminar at UCLA; see \cite{F2024CULA}. About a month after this seminar talk, Hagen Papenburg posted the preprint \cite{P24} on the arXiv, in which he develops an alternative approach to studying dispersive PDEs with quasi-periodic initial data.

(b)~For the local results of the standard nonlinear Schr\"odinger equation \eqref{cnls} in this paper and the derivative nonlinear Schr\"odinger equation \eqref{dnls} in \cite{DLX24II}, we do not need any other additional conditions for the frequency vector $\omega$ but the {non-resonance} condition introduced above, as there is no {\bf small divisor problem} appearing in the {\bf quasi-periodic motion},
%that is, time quasi-periodic solutions to ODEs, PDEs, Hamiltonian systems {and so on and so forth}.
encompassing time quasi-periodic solutions to ODEs, PDEs, Hamiltonian systems, and more.

(c)~The local analysis of \eqref{cnls} works for arbitrary $2\leq p\in\mathbb N$, especially including the mass-critical case (i.e., $p=2$); see \cite{Merle2005} and the references therein.

\end{rema}

{\bf Outline of Proof}\quad We implement the steps in the following diagram to prove our main results.
\[
\xymatrix{
\boxed{\substack{\text{Magentauction of a nonlinear PDE with spatial Fourier series }\\[0.5mm]\text{to a nonlinear system of infinite coupled ODE}}}
\ar[d]^{\text{feedback of nonlinearity}}\\
\boxed{\text{Picard iteration}}
\ar[d]
\ar[r]_{
\substack{
\text{alternating}\\[0.5mm]
\text{higher-dimensional}\\[0.5mm]
\text{discrete convolution}
}
}^{\substack{\text{{pcc}}\\[0.5mm] \text{Feynman diagam}}}&
\boxed{\text{Combinatorial tree}}\ar[d]
\\
\boxed{\text{Cauchy sequence}}\ar[d]&
\boxed{\text{Exp/poly decay}}
\ar[l]_{\text{interpolation}}\ar[d]
\\
\boxed{\text{Local existence}}\ar[d]^{?}&\boxed{\text{Uniqueness}}\\
\boxed{\text{Global problem}} &
}
\]
An approach of this nature has been previously applied in other settings in \cite{DG2017JAMS, DLX2021, DLX22AR}.

%This paper is the first part of this plan, that is, to study the standard nonlinear Schr\"odinger eqution \eqref{cnls} with quasi-periodic initial data \eqref{id0}.

%This paper constitutes the first part of the two-paper series, focusing on the study of the standard nonlinear Schr\"odinger equation \eqref{cnls} with quasi-periodic initial data \eqref{id0}.

%For the convenience of the reader and to make this paper as self-contained and systematic as possible, we give most of the proofs in an intuitively clear, accessible, coherent manner.

%Some of them (similar to the analysis in \cite{DG2017JAMS,DLX21AR,DLX22AR}) are put in Appendix, and the rest are put in the body.

\section{Preliminaries}\label{pre}

In this section, we first introduce the concept of (spatially) almost/quasi-periodic functions in \autoref{sqf}, which will serve as our initial data/solutions. Next, to deal with the alternating discrete convolution of higher dimensions, we propose the power of $\ast^{[\cdot]}$ for labelling the complex conjugate (pcc for short) appearing in the nonlinear part of the Picard iteration in \autoref{lc}. Furthermore, we combine the multi-linear operator and the alternating sum condition with some basic concepts in \autoref{scc}. These will be used in the next paper \cite{DLX24II} as well.
%,
%and introduce the Feynman diagram in

%the power of $\ast^{[\cdot]}$ to label the complex conjugate appearing in the Picard iteration in the sense of combination.
\subsection{(Spatially) Almost-Periodic Functions}\label{sqf}
In this subsection we introduce the concept of (spatially) almost periodic functions, which will be our initial data/solutions.
\begin{defi}\cite{O2015SIAMMA,CL20JDE}
We say that a bounded continuous function $f:\mathbb R\rightarrow\mathbb C$ is almost periodic if it satisfies one of the following statements:
\begin{itemize}
  \item[(i)] Bohr (1925) defined a uniformly almost periodic function $f$ as an element in the closure of the trigonometric polynomials w.r.t. uniform norm, or rather, for every $\epsilon>0$, there exists a trigonometric polynomial $P_\epsilon$ (a finite linear combination of sine and cosine waves), such that the distance w.r.t. between $f$ and $P_\epsilon$ is less than $\epsilon$, that is,
      \[\sup_{x\in\mathbb R}|f(x)-P_\epsilon(x)|<\epsilon.\]

  \item[(ii)] Bochner (1926) proved that Bohr's definition is equivalent to the following: Given a sequence $\{x_n\}_{n\in\mathbb N}\subset\mathbb R$, the collection $\{f(\cdot+x_n)\}_{n\in\mathbb N}$ is precompact in $L^\infty(\mathbb R)$. Namely, there exists a subsequence $\{f(\cdot+x_{n_j})\}$ uniformly convergent on $\mathbb R$.
 % \item For every $\epsilon>0$, there exists a positive constant $L=L(\epsilon,f)>0$ such that every interval of length $L$ on $\mathbb R$ contains a number $\tau$ such that
%\[\sup_{x\in\mathbb R}|f(x-\tau)-f(x)|<\epsilon.\]
%  \item [(ii)]
 % \item [(iii)]For any $\epsilon>0$, there exists a trigonometric polynomial $P_\epsilon(x)$ such that

  \item [(iii)] There are a frequency vector $\omega\in\mathbb T^\nu$ and a continuous $F:\mathbb T^\nu\rightarrow\mathbb R$,  where $\nu\in\mathbb N\cup\{\infty\}$, such that
      \[f(x)=F([\omega x]),\]
where $[\cdot]: \mathbb R^\nu\rightarrow\mathbb T^\nu$ is a mapping defined by letting $[y]\equiv y~(\text{mod}~2\pi)$.
\end{itemize}
\end{defi}

It follows from the torus definition of almost periodic functions that $f$ has the following formal Fourier expansion:
\[f(x)=\sum_{\lambda\in\Lambda}\hat f(\lambda)e^{{\rm i}\lambda x},\quad x\in\mathbb R.\]
Here, we refer to $\lambda$ as the Fourier index, $\Lambda$ as the Fourier support set of $f$, and $\{\hat f(\lambda)\}_{\lambda\in\Lambda}$ as a family of Fourier coefficients, defined by
\begin{align*}
\hat f(\lambda)
%&=\langle f,e^{{\rm i}\lambda x}\rangle_{\mathcal L^2(\mathbb R)}\\
%&=\mathcal M(fe^{-{\rm i}\lambda x})\\
&=\lim_{L\rightarrow+\infty}\frac{1}{2L}\int_{-L}^{L}f(x)e^{-{\rm i}\lambda x}{\rm d}x,\quad\lambda\in\Lambda.
\end{align*}
Here $\mathcal M$ represents the mean value of a function over the real line $\mathbb R$, defined as:
\[\mathcal M(f):=\frac{1}{2L}\int_{-L}^{+L}f(x){\rm d}x.\]
And the inner product $\langle\cdot,\cdot\rangle_{\mathcal L^2(\mathbb R)}$ is defined as follows:
\begin{align*}
\langle f,g\rangle_{\mathcal L^2(\mathbb R)}:&=\mathcal M(f\overline{g})\\
&=\lim_{L\rightarrow+\infty}\frac{1}{2L}\int_{-L}^{+L}f(x)\overline{g(x)}{\rm d}x.
\end{align*}
%\begin{defi}\label{sqpf}
%For two functions $f_1:\bR\rightarrow\bC$ and $f_2:\bR\rightarrow\bC$, define their  inner product (in the mean sense) by letting
%\[
%\langle f_1,f_2\rangle_{L^2(\bR)}:=\lim_{L\rightarrow\infty}\frac{1}{2L}\int_{-L}^{L}f_1(x)\overline{f_2(x)}{\rm d}x.
%\]
%\end{defi}

%According to the Fourier support set of $f$, almost periodic functions correspond to different classes of functions. If $\Lambda\subset\omega\mathbb Z$, then they are periodic functions; if $\Lambda\subset\sum_{j=1}^{\nu}\omega_j\mathbb Z$, where $2\leq\nu<\infty$, then they are quasi-periodic functions; if $\Lambda\subset\sum_{j=1}^{\infty}\omega_j\mathbb Z$, then they are almost-periodic functions. Especially we pay attention to the nontrivial and finite case, that is, quasi-periodic functions, that is, functions have multiple periods whose frequencies are rationally independent. For example
%$$f(x)=\cos x+\cos\omega x,$$ where $\omega$ is an irrational number and $x\in\mathbb R$.
%From this simple example, we can see that the image of such a quasi-periodic function can be embedded into  a higher-dimensional space using the so-called generating/hull/torus function $F(x,y)=\cos x+\cos y$ (e.g., from line $(x;f)$ to surface $(x,y;F)$, i.e., such a line $(x;f)$ can be embedded into a surface $(x,y;F)|_{y=\omega x}$). Clearly, $F(\cdot,\cdot)$ is $2\pi$-periodic w.r.t. each direction, and $f(x)=F(x,y)|_{y=\omega x}$.
The Fourier support set of $f$ categorizes almost periodic functions into different classes based on their properties. Specifically:
\begin{itemize}
  \item If $\Lambda\subset\omega\mathbb Z$, then the functions are periodic.
  \item if $\Lambda\subset\sum_{j=1}^{\nu}\omega_j\mathbb Z$, where $2\leq\nu<\infty$, they are quasi-periodic;
  \item if $\Lambda\subset\sum_{j=1}^{\infty}\omega_j\mathbb Z$, they are almost-periodic functions.
\end{itemize}
We are particularly interested in the nontrivial and finite case, namely quasi-periodic functions, where functions have multiple periods with frequencies that are rationally independent. For instance,
$$f(x)=\cos x+\cos\omega x,$$ where $\omega$ is an irrational number and $x\in\mathbb R$.
From this example, we observe that the image of such a quasi-periodic function can be represented in a higher-dimensional space using the so-called generating/hull/torus function
$F(x,y)=\cos x+\cos y$ (e.g., from line $(x;f)$ to surface $(x,y;F)$, i.e., such a line $(x;f)$ can be embedded into a surface $(x,y;F)|_{y=\omega x}$). Clearly, $F(\cdot,\cdot)$ is $2\pi$-periodic w.r.t. each direction, and $f(x)=F(x,y)|_{y=\omega x}$.

Regarding quasi-periodic functions, we assume that $\omega=(\omega_j)_{1\leq j\leq\nu}\in\mathbb R^\nu$ is assumed to be non-resonant, implying that $\{e^{{\rm  i}\langle n\rangle x}\}$ is orthonormal; see Proposition \ref{basi}.

\begin{prop}\label{basi}
%Clearly, it follows from the {non-resonance} condition {on} the frequency $\omega$ that
If the frequency vector $\omega$ is non-resonant, then
\[
\langle e^{{\rm i}\langle n\rangle x},e^{{\rm i}\langle n^\prime\rangle x}\rangle_{\mathcal L^2(\mathbb R)}={\delta_{0,n-n^\prime}},\quad \forall n,n^\prime\in\bZ^\nu.
\]
\end{prop}
\begin{proof}
%\subsubsection{Proof of \eqref{ooo}}\label{pooo}
Let $n$ and $n^\prime$ be two elements in $\bZ^\nu$.

If $n=n^\prime$, then $\langle n-n^{\prime}\rangle=0$. By the definition of $\langle\cdot,\cdot\rangle_{L^2(\bR)}$, we have
\begin{align*}
\langle e^{{\rm i}\langle n\rangle x},e^{{\rm i}\langle n^\prime\rangle x}\rangle_{\mathcal L^2(\bR)}&=\lim_{L\rightarrow\infty}\frac{1}{2L}\int_{-L}^{L}e^{{\rm i}\langle n-n^\prime\rangle x}{\rm d}x=1.
%&=\lim_{L\rightarrow\infty}\frac{1}{2L}\int_{-L}^{L}{\rm d}x\\
%&=1.
\end{align*}

If $n\neq n^\prime$, it follows from the non-{resonance} condition on the frequency $\omega$ that $\langle n-n^\prime\rangle\neq0$. Furthermore, according to the definition of $\langle\cdot,\cdot\rangle_{L^2(\bR)}$, one can derive that
\begin{align*}
\langle e^{{\rm i}\langle n\rangle x},e^{{\rm i}\langle n^\prime\rangle x}\rangle_{\mathcal L^2(\bR)}&=\lim_{L\rightarrow\infty}\frac{1}{2L}\int_{-L}^{L}e^{{\rm i}\langle n-n^\prime\rangle x}{\rm d}x=0.
%&=\lim_{L\rightarrow\infty}\frac{1}{2L}\int_{-L}^{L}\cos(\langle n-n^\prime\rangle x){\rm d}x\\
%&=\lim_{L\rightarrow\infty}\frac{\sin\left(\langle n-n^\prime\rangle L\right)}{\langle n-n^\prime\rangle L}\\
%&=0.
\end{align*}

This completes the proof of Proposition \ref{basi}.
\end{proof}

To sum up, for a quasi-periodic function $f$ with frequency vector $\omega$, it has the following Fourier series
\[f(x)=\sum_{n\in\mathbb Z^\nu}\hat f(n)e^{{\rm i}\langle n\rangle x}, \]
where
\[\hat f(n)=\lim_{L\rightarrow\infty}\frac{1}{2L}\int_{-L}^{+L}f(x)e^{-{\rm i}\langle n\rangle x},\quad n\in\bZ^\nu.\]

It should be pointed out that (i)~such functions don't decay to zero (oscillating at infinity); (ii)~they are not periodic
and cannot therefore be studied on a circle. In the decaying/periodic cases, various (finite) $L^2_0(\mathbb R)/L^2(\mathbb T)$ conserved quantities can help us to analyze the global Cauchy problem, but, in the quasi-periodic setting, even if averaged, $L^2(\mathbb R)$ conserved quantities do ``not" (or rather, we do not know how to use them); see \cite{K2023}.

%\marginpar{\color{magenta}We add a simple introduction to spatially quasi-periodic functions in this subsection. So we modify this subsection entitled ``Frequency vector" into ``Spatially quasi-periodic function", which will be used in both cNLS and dNLS. }

%Let $\omega=(\omega_j)_{1\leq j\leq \nu}\in\bR^\nu$ be a given frequency vector.
%
%The standard Euclidean inner product with respect to this frequency vector $\omega$ is denoted by $\langle n\rangle$, that is to say,
%\[\langle n\rangle\triangleq n\cdot\omega:=\sum_{j=1}^\nu n_j\omega_j,\quad \forall n\in\bZ^\nu.\]
%
%Throughout this paper, the frequency vector $\omega$ is non-resonant or rationally independent, that is, $\langle n\rangle=0$ implies that $n=0\in\bZ^\nu$.

%\marginpar{Some references on spatial quasi-periodicity are added here; see footnote. If you have more suggested references, please add them here. Thanks.\\
%By the way, we think that the footnote part of introduction to spatial quasi-periodicity seems not good. So could you please re-write it? Thanks.}
%For the introduction of
%{\bf spatially quasi-periodic functions}
By the way, regarding the spatial quasi-periodicity, it is related to the study of quasi-crystals in materials\cite{Baake02,deift2008,BDG2016NAMS,deift2017}, quasi-patterns in the Faraday wave experiment\cite{BIS2017CMP,iooss2019}, rogue waves in oceanography\cite{WZ2021JFM,WZ2021JNS,st2022}, Bose-Einstein condensation in quantum mechanics\cite{W2020CMP}, the theory of conductivity \cite{DN2005UMN}, irrational tori in mathematics \cite{B2007} and so on and so forth.

%There are some important works in PDE theory based on spatial quasi-periodic setting, such as, KdV\cite{T2012SIAMMA,DG2017JAMS,BDGL2018DUKE,EVY2019
%TAMS,LY2020JFA}, NLS\cite{O2015SIAMMA,W2020CMP}, and the references therein.

% defined on the real line $\bR$, we first give the definition of inner product in the mean sense.
%; see Definition \ref{sqpf}.

%Now we give the definition of spatially quasi-periodic functions on the real line $\bR$; see Definition \ref{defisqp}.
%\begin{defi}\label{defisqp}
%A function $f:\bR\rightarrow\bC$ is said to be spatially quasi-periodic with the frequency vector $\omega$ if it has the following Fourier series
%\[
%f(x)=\sum_{n\in\mathbb Z^\nu}\hat f(n)e^{{\rm i}\langle n\rangle x},
%\]
%where the Fourier coefficient $\hat f$ is defined by letting
%\[
%\hat f(n)=\langle f,e^{{\rm i}\langle n\rangle x}\rangle_{L^2(\bR)}=\lim_{L\rightarrow\infty}\frac{1}{2L}\int_{-L}^{L}f(x)e^{-{\rm i}\langle n\rangle x}{\rm d}x,\quad \forall n\in\bZ^\nu.
%\]
%\end{defi}
%Throughout this paper, we will use the following concepts and notations.

%{\color{magenta}
%\marginpar{We think the new name ``alternating discrete convolution" is better than the discrete convolution. It's the core of this paper, so we give its strict definition here.}

\subsection{Power of $\ast^{[\cdot]}$ for the Complex Conjugate}\label{lc}

In this subsection we introduce the power of $\ast^{[\cdot]}$ for labelling the complex conjugate (pcc for short) to deal with the alternating discrete convolution of higher dimensions appearing in the Picard iteration.
\begin{defi}
The {\bf alternating discrete convolution of higher dimensions} for complex functions $f_j: \copyright\rightarrow\bC$, where $j=1,\cdots,Q\in2+\bN$, with total distance $n\in\copyright$, is defined by letting
\begin{align}\label{adc}
f_1\star\cdots\star f_Q(n):=\sum_{\substack{n_j\in\copyright,~~j=1,\cdots,Q\\{\sum_{j=1}^{Q}(-1)^{j-1}n_j=n}}}
\prod_{j=1}^Q\{f_j(n_j)\}^{\ast^{[j-1]}}.
\end{align}
Here $\star$ stands for the alternating discrete convolution operation of higher dimensions, $Q$ will be $2p\sigma$ in the following sections, $\copyright$ stands for the basic lattice space, i.e., $\mathbb Z^\nu$, %$\copyright^Q$ is the $Q$-times Cartesian product of $\copyright$, $\cc\ca\cs$ is the combinatorial alternating sum defined by letting $\cc\ca\cs(n^{(Q)})=\sum_{j=1}^Q(-1)^{j-1}n_j$, where $n^{(Q)}=(n_j)_{1\leq j\leq Q}\in\copyright^Q$ (see also Definition \ref{defias}),
and ${\ast^{[\bullet]}}$ denotes the power of the complex conjugate appearing in the Picard iteration (see \autoref{lc}).
\end{defi}
%\begin{rema}
%Here alternating means that the pattern of $+-+\cdots+-+$ generating from the complex nonlinearity.
%\end{rema}

Taking the alternating discrete convolution of higher dimensions as a guide, we propose the power of $\ast^{[\cdot]}$ to label the complex conjugate (see \autoref{lc}), and some combinatorial concepts and notations (see \autoref{scc}) to deal with the complicated Picard iteration in a combinatorial manner.
%}

The alternating pattern of $+-+\cdots+-+$ generating from the complex-valued nonlinearity prompts us to consider the so-called {\bf ``power of $\ast^{[\cdot]}$"}, which is {defined} as follows: for any complex number $z\in\bC$, we use $z^{\ast^0}$ and $z^{\ast^1}$ to stand for itself $z$ and its complex conjugate $\bar z$ {(here ``~$\bar{~}$~" denotes the complex conjugate operation as usual)}, that is,
\begin{align}\label{c}
z\triangleq z^{\ast^0}\quad\text{and}\quad\bar z\triangleq z^{\ast^1}.
\end{align}
Here $\ast$ not only can be viewed as a unifier of $+$ and $-$ used in \cite{2022not}, but also can help us to express the complex conjugate in a manner of combinatorics during the Picard iteration. Regarding this notation, we have the following several operation properties:

\begin{itemize}
  \item For any given $m\in\bN$,
  \begin{align}\label{co1}
  \text{the result of}~m\text{-times complex conjugate of}~z~\text{is equal to}~z^{\ast^{[m]}},
  \end{align}
  where $[\cdot]\in\{0,1\}$ is determined by the congruence equation $[m]\equiv m~(\text{mod}~2)$.
  %; see \autoref{secco1}.
 \item For any given $m\in\bN$, we have
 \begin{align}\label{co2}
 \overline{z^{\ast^{[m]}}}=z^{\ast^{[m+1]}}.
 \end{align}
% see \autoref{secco2}.
 \item For any given $m,m^\prime\in\bN$, we have
 \begin{align}\label{co3}
 \left(z^{\ast^{[m]}}\right)^{\ast^{[m^\prime]}}=
 z^{\ast^{[m+m^\prime]}}.
 \end{align}
 %see \autoref{secco3}.
 % (the total conjugacy of $m$ and $m^\prime$ times conjugacy of $z$ is $m+m^\prime$-times conjugacy of $z$).
 \item For any given $m\in\bN$, complex numbers $z_1\in\bC$ and $z_2\in\bC$, we have
 \begin{align}\label{lin}
 (z_1+z_2)^{\ast^{[m]}}=z_1^{\ast^{[m]}}+z_2^{\ast^{[m]}}
 \end{align}
 and
 \[
 (z_1z_2)^{\ast^{[m]}}=z_1^{\ast^{[m]}}\cdot z_2^{\ast^{[m]}}.
 \]
 %see \autoref{seclin}.
\end{itemize}
Proofs of \eqref{co1}--\eqref{lin} can be given by induction.
%\subsection{Proofs in \autoref{lc}}

%Furthermore, we give an illustration of \eqref{co1} and proofs for \eqref{co2}, \eqref{co3} and \eqref{lin} in the following subsubsections.
%\marginpar{\color{magenta}We delete the proof for these simple facts.}

%\subsection{Alternating discrete convolution of higher dimensions}

\subsection{Combinatorial Structure and Some Basic Concepts}\label{scc}
%\marginpar{{\color{magenta}Fei: Is it better to use {\em Combinatorization} rather than {\em Combination}? (Compare the use of normalization, linearization etc.)} {\color{blue}I feel that neither of these words exist or are common in the English language. I have tried to use the term ``combinatorial structure''.}}
In this subsection, we address the combinatorial structure of the domain of the multi-linear operator and the summation condition; see \eqref{adc}.
%\marginpar{Based on your last suggestion, we make these concepts to be definitions together with propositions. We think that they seems to be much more strict and systematic.}

Throughout this paper,
%which will be used in the nonlinear Schr\"odinger equations \eqref{cnls} and \eqref{gdnls},
for any given $p\in1+\bN$,
denote by $P=2p+1$.
\begin{defi}%(Branch Set)
The {\bf branch set} $\Gamma^{(k)}$ is defined by letting
\begin{align}\label{g}
\Gamma^{(k)}=
\begin{cases}
\{0,1\},&k=1;\\
\{0\}\cup(\Gamma^{(k-1)})^{P},&k\geq2.
\end{cases}
\end{align}
%{\color{magenta}In addition, any element $\gamma^{(k)}$ in the branch set $\Gamma^{(k)}$ is said to be a {\bf branch} of $\Gamma^{(k)}$. }
\end{defi}

%{\color{magenta}
\begin{rema}
The function of the branch set is used to label or follow every term w.r.t. the initial data in the Picard iteration. This is also the beginning of the combinatorial analysis applied to the Cauchy problem of the nonlinear infinite system of coupled ODEs in a way of the alternating discrete convolution of higher dimensions.
\end{rema}
Some essential concepts and useful notations related to the alternating discrete convolution of higher dimensions in the Picard iteration will be introduced. On each branch, we will define the first counting function $\sigma$ for the numbers of the initial data (see Definition \autoref{cf}), the second counting function $\ell$ for the times of the alternating discrete convolution of higher dimensions (see Definition \ref{cf2}), the combinatorial lattice space $\cN$ (see Definition \ref{cls}),  the combinatorial alternating sums $\cc\ca\cs$ (see Definition \ref{defias}).

First, on each branch, we define the counting functions $\sigma$ and $\ell$.
%}
\begin{defi}%(Counting Function)
\label{cf}
The {\bf first counting function} $\sigma$ ($2p\sigma$
indeed) acting on the branch set is defined by letting
\begin{align}\label{s}
\sigma(\gamma^{(k)})=
\begin{cases}
\frac{1}{2p},&\gamma^{(k)}=0\in\Gamma^{(k)},k\geq1;\\
\frac{P}{2p},&\gamma^{(1)}=1\in\Gamma^{(1)};\\
\sum_{j=1}^{P}\sigma(\gamma_j^{(k-1)}),&\gamma^{(k)}=(\gamma_j^{(k-1)})_{1\leq j\leq P}\in(\Gamma^{(k-1)})^{P},k\geq2.
\end{cases}
\end{align}
\end{defi}

%{\color{magenta}
\begin{defi}\label{cf2}
The {\bf second counting function} $\ell$ ($2p\ell$ indeed) acting on the branch set is defined by letting
\begin{align}\label{ee}
\ell(\gamma^{(k)})=
\begin{cases}
0,&\gamma^{(k)}=0\in\Gamma^{(k)},k\geq1;\\
1,&\gamma^{(1)}=1\in\Gamma^{(1)};\\
1+\sum_{j=1}^{P}\ell(\gamma_j^{(k-1)}),&\gamma^{(k)}=(\gamma_j^{(k-1)})_{1\leq j\leq P}\in(\Gamma^{(k-1)})^{P},k\geq2.
\end{cases}
\end{align}
\end{defi}

\begin{rema}\label{itu}
There is an intuition for these two counting functions $\sigma$ and $\ell$.
In fact, $2p\sigma$ depicts the degree/multiplicity of nonlinearity in the sense of combinatorics, that is, the number of the initial Fourier data on each branch in the Picard iteration; and $2p\ell$ stands for the number of integrations.
\end{rema}

With Remark \ref{itu}
in mind, we can directly obtain the relation between $\sigma$ and $\ell$ (see Proposition \ref{propsl}), and the parity of $\sigma$ and $\ell$ (see Proposition \ref{lemms}).
%}

\begin{prop}\label{propsl}
For all $k\geq1$, we have
\begin{align}\label{ell}
\sigma(\gamma^{(k)})=\ell(\gamma^{(k)})+\frac{1}{2p}.
\end{align}
\end{prop}
%\subsubsection{Proof of \eqref{ee}}\label{eee}
\begin{proof}
It is obvious that \eqref{ell} holds for all $0=\gamma^{(k)}\in\Gamma^{(k)}$, where $k\geq1$, and $1=\gamma^{(1)}\in\Gamma^{(1)}$. This shows that \eqref{ell} holds for $k=1$.

Let $k\geq2$. Assume that \eqref{ell} is true for all $1<k^\prime<k$.

For $(\gamma_j^{(k-1)})_{1\leq j\leq P}=\gamma^{(k)}\in(\Gamma^{(k-1)})^{P}$,
it follows from the definitions \eqref{s} and \eqref{ee}, and the induction hypothesis that
\begin{align*}
\sigma(\gamma_j^{(k-1)})=&\sum_{j=1}^P\sigma(\gamma_j^{(k-1)})\\
=&\sum_{j=1}^P\left(\ell(\gamma_j^{(k-1)})+
\frac{1}{2p}\right)\\
%=&1+\sum_{j=1}^P\ell(\gamma_j^{(k-1)})+
%\frac{1}{2p}\\
=&\ell(\gamma^{(k-1)})+\frac{1}{2p}.
\end{align*}
This proves that \eqref{ell} holds for $k$, and hence for all $k\geq1$ by induction.
This completes the proof of Proposition \ref{propsl}.
\end{proof}

%Regarding the counting function $\sigma$, there is another function $\ell$, associated with the time integral $\int_0^t$, defined by letting

%For $\ell(\gamma^{(k)})$, where $k\geq2$, we have the following elementary equality
%\begin{align}\label{ee}
%\ell(\gamma^{(k)})=1+\sum_{j=1}^{P}\ell(\gamma_j^{(k-1)}),\quad \forall(\gamma_j^{(k-1)})_{1\leq j\leq P}=\gamma^{(k)}\in(\Gamma^{(k-1)})^{P};
%\end{align}
%see  \autoref{eee} in Appendix.

%With Remark \ref{itu} in mind, we have Proposition \ref{prop} and Lemma \ref{lemms}.

\begin{prop}\label{lemms}
For all $k\geq1$, $2p\sigma(\gamma^{(k)})$ is odd and $2p\ell(\gamma^{(k)})$ is even on each branch $\gamma^{(k)}\in\Gamma^{(k)}$.
\end{prop}
\begin{proof}
It's obvious that $2p\sigma(0)=1$ and $2p\sigma(1)=P$. This shows that this conclusion is true for $k=1$.

Let $k\geq2$. Assume that $2p\sigma(\gamma^{(k^\prime)})$ is odd for all $1<k^\prime<k$.

For $k$, we have $2p\sigma(0)=1$. In addition, for $\gamma^{(k)}=(\gamma_j^{(k-1)})_{1\leq j\leq P}$, by the definition \eqref{s} of $\sigma$, one can derive that
\[2p\sigma(\gamma^{(k)})=2p\sigma(\gamma_1^{(k-1)})+\cdots+2p\sigma(\gamma_{P}^{(k-1)}).\]
It follows from the induction hypothesis that $2p\sigma(\gamma_j^{(k-1)})$ is odd for all $j=1,\cdots,P$. Hence $2p\sigma(\gamma^{(k)})$ can be viewed as {sums} of $P$ odd numbers, this implies that $2p\sigma(\gamma^{(k)})$
 is odd.

By induction, we prove that $2p\sigma(\gamma^{(k)})$ is odd for all $k\geq1$.

It follows from the above proof and the relation \eqref{propsl} between $\sigma$ and $\ell$ that $2p\ell$ is even. This completes the proof of Proposition \ref{lemms}.
\end{proof}
%\marginpar{So we make it to be a Proposition and give the proof here.}

\begin{rema}
It should be emphasized that Proposition \ref{lemms} will play an essential role in the proofs of Lemma~\ref{lemas} and Lemma~\ref{ca}.
%, though it seems simple.
\end{rema}

Next, on each branch, we introduce the domain of the alternating discrete convolution of higher dimension, that is, the so-called combinatorial lattice space (see Definition \ref{cls}), and the function, associated with the alternating discrete convolution of higher dimension, defined on the combinatorial lattice space, that is, the so-called combinatorial alternating sums (see Definition \ref{defias}).

\begin{defi}%(Combinatorial Lattice Space)
\label{cls}
The {\bf {combinatorial} lattice space} $\cN^{(k,\gamma^{(k)})}$ {originated from} $\bZ^\nu$ on each branch is defined by letting
\begin{align}\label{n}
&\cN^{(k,\gamma^{(k)})}=
\begin{cases}
\bZ^\nu,&\gamma^{(k)}=0\in\Gamma^{(k)},k\geq1;\\
(\bZ^\nu)^{P},&\gamma^{(1)}=1\in\Gamma^{(1)};\\
\prod_{j=1}^{P}\cN^{(k-1,\gamma_j^{(k-1)})},&\gamma^{(k)}=(\gamma_j^{(k-1)})_{1\leq j\leq P}\in(\Gamma^{(k-1)})^{P},k\geq2.
\end{cases}
\end{align}
\end{defi}

\begin{defi}\label{defias}
%The last but most important combinatorial concept is the so-called
The {\bf combinatorial alternating sum{s}}, denoted by $\cc\ca\cs(n^{(k)})$, of $n^{(k)}\in\cN^{(k,\gamma^{(k)})}$, is defined by letting
\begin{align}\label{as1}
&\cc\ca\cs(n^{(k)})=
\begin{cases}
n^{(k)},&\gamma^{(k)}=0\in\Gamma^{(k)},n^{(k)}\in\cN^{(k,0)}, k\geq1;\\
\sum_{j=1}^{P}(-1)^{j-1}n_j,&\gamma^{(1)}=1\in\Gamma^{(1)},n^{(1)}=(n_j)_{1\leq j\leq P}\in(\bZ^\nu)^{P};\\
\sum_{j=1}^{P}(-1)^{j-1}\cc\ca\cs(n_j^{(k-1)}),&\gamma^{(k)}=(\gamma_j^{(k-1)})_{1\leq j\leq P}\in(\Gamma^{(k-1)})^{P},\\
&n^{(k)}=(n_j^{(k-1)})_{1\leq j\leq {P}}\in\prod_{j=1}^{P}\cN^{(k-1,\gamma_j^{(k-1)})}, k\geq2.
\end{cases}
\end{align}
\end{defi}

As to the combinatorial lattice space and the combinatorial alternating sums, from the point of the generating lattice space $\bZ^\nu$, we know the number of components in $\cN$ (see Proposition \ref{2ps}), and the shape of $\cc\ca\cs$ (see Proposition \ref{lemas}).

\begin{prop}\label{2ps}
Let dim$_{\bZ^{\nu}}\cN^{(k,\gamma^{(k)})}$ be {\bf the number of components in $\cN^{(k,\gamma^{(k)})}$ per $\bZ^\nu$}. Then
\begin{align}\label{dimsi}
\dim_{\bZ^\nu}\cN^{(k,\gamma^{(k)})}=
2p\sigma(\gamma^{(k)}),\quad\forall k\geq1.
\end{align}
That is,
\[\cN^{(k,\gamma^{(k)})}=(\mathbb Z^\nu)^{2p\sigma(\gamma^{(k)})}.\]
\begin{proof}
For $\gamma^{(k)}=0\in\Gamma^{(k)}$, where $k\geq1$, it follows from $\cN^{(k,0)}=\bZ^\nu$, the definition of $\dim_{\bZ^\nu}\cN$, and the definition \eqref{s} of $\sigma$ that $\dim_{\bZ^\nu}\cN^{(k,0)}=1=2p\sigma(0)$. This shows that \eqref{dimsi} holds for all $\gamma^{(k)}=0\in\Gamma^{(k)}$, where $k\geq1$.

For $\gamma^{(1)}=1\in\Gamma^{(1)}$, $\cN^{(1,1)}=(\bZ^\nu)^P$ implies that $\dim_{\bZ^\nu}\cN^{(1,1)}=P=2p\sigma(1)$. Hence \eqref{dimsi} is true for $\gamma^{(1)}\in\Gamma^{(1)}$.

For $k\geq2$. Assume that \eqref{dimsi} holds for all $1\leq k^\prime<k$.

For $k$ and $\gamma^{(k)}=(\gamma_j^{(k-1)})_{1\leq j\leq P}\in(\Gamma^{(k-1)})^{P}$, by the definitions of $\cN, \dim_{\bZ^\nu}\cN, \sigma$ and induction hypothesis, one can derive that
\begin{align*}
\dim_{\bZ^\nu}\cN^{(k,\gamma^{(k)})}
&=\sum_{j=1}^{P}\dim_{\bZ^\nu}\cN^{(k-1,\gamma_j^{(k-1)})}\\
&=\sum_{j=1}^{P}2p\sigma(\gamma_j^{(k-1)})\\
&=2p\sigma(\gamma^{(k)}).
\end{align*}
This proves that \eqref{dimsi} is true for $k$.

By induction, \eqref{dimsi} holds for all $k\geq1$. This completes the proof of Proposition \ref{2ps}.
\end{proof}

\end{prop}
%times of evolution

%It follows from the power-type nonlinearity and complex conjugacy that both the product of $c$ and $\bar c$  and the restriction of sum are alternative. This yields the following concept.

\begin{rema}
By Proposition \ref{2ps}, for any $k\geq1$, it is reasonable to set $n^{(k)}=(m_j)_{1\leq j\leq 2p\sigma(\gamma^{(k)})}$, where $m_j\in\bZ^\nu$ for all $j=1,\cdots,P$.
\end{rema}

%Regarding the alternating sum $\cc\ca\cs$, it has an equality.
%It follows from induction that
%\arginpar{There is something wrong with the original proof. I re-proved it in the original position.}
\begin{prop}\label{lemas}For all $k\geq1$, we have
\begin{align}\label{as2}
\cc\ca\cs(n^{(k)})=\sum_{j=1}^{2p\sigma(\gamma^{(k)})}(-1)^{j-1}m_j.
\end{align}
\end{prop}
\begin{proof}
%[Proof of \eqref{as2}]%\marginpar{\color{magenta}We make it to be a lemma.}
For $\gamma^{(k)}=0\in\Gamma^{(k)}$, set $m_1=n^{(k)}\in\cN^{(k,0)}$, where $k\geq1$. In this case, by definition \eqref{s}, we know that $2p\sigma(0)=1$. Furthermore, it follows from definition \eqref{as1} that \begin{align*}
\cc\ca\cs(n^{(k)})&=m_1\\
&=\sum_{j=1}^{2p\sigma(0)}(-1)^{j-1}m_j.
\end{align*} This shows that \eqref{as2} holds for all $\gamma^{(k)}=0\in\Gamma^{(k)}$, where $k\geq1$.

For $\gamma^{(1)}=1\in\Gamma^{(1)}$, set $(m_j)_{1\leq j\leq P}=n^{(1)}\in(\bZ^\nu)^{P}$. By definition \eqref{s}, we know that $2p\sigma(1)=P$. It  follows from definition \eqref{as1} that \begin{align*}\cc\ca\cs(n^{(1)})&=\sum_{j=1}^{P}(-1)^{j-1}m_j\\
&=\sum_{j=1}^{2p\sigma(1)}(-1)^{j-1}m_j.
\end{align*} Hence \eqref{as2} holds true for $k=1$.

For $k\geq2$. Assume that \eqref{as2} is true for all $1<k^\prime<k$.

For $k$, we consider only the remaining case, that is, $\gamma^{(k)}=(\gamma_j^{(k-1)})_{1\leq j\leq P}\in(\Gamma^{(k-1)})^{P}$ and $n^{(k)}=(n_j^{(k-1)})_{1\leq j\leq P}\in\prod_{j=1}^{P}\cN^{(k-1,\gamma^{(k-1)})}$, where
\begin{align*}
&n_1^{(k-1)}=(m_{j})_{1\leq j\leq2p\sigma(\gamma_1^{(k-1)})};\\
%&\vdots\\
&n_{j^\prime}^{(k-1)}=(m_{j})_{2p\sum_{j=1}^{j^\prime-1}\sigma(\gamma_j^{(k-1)})+1\leq j\leq2p\sum_{j=1}^{j^\prime}\sigma(\gamma_j^{(k-1)})},\quad j^\prime=2,\cdots,P.
%&\vdots\\
%&n_{P}^{(k-1)}=(m_j)_{2p\sum_{j=1}^{2p}\sigma(\gamma_j^{(k-1)})+1\leq j\leq2p\sigma(\gamma^{(k)})}.
\end{align*}
By the definitions of \eqref{as1} and \eqref{s}, and Lemma \ref{lemms}, one can derive that
%\marginpar{I think that there is something wrong with the original proof. So I re-prove it. }
\begin{align*}
\cc\ca\cs(n^{(k)})
&=\sum_{j^\prime=1}^{P}(-1)^{j^\prime-1}
\cc\ca\cs(n_{j^\prime}^{(k-1)})\\
&=\cc\ca\cs(n_1^{(k-1)})+\sum_{j^\prime=2}^{P}(-1)^{j^\prime-1}\cc\ca\cs(n_{j^\prime}^{(k-1)})\\
&=\sum_{j=1}^{2p\sigma(\gamma_1^{(k-1)})}(-1)^{j-1}m_j+
\sum_{j^\prime=2}^{P}(-1)^{j^\prime-1}\sum_{j=1}^{2p\sigma(\gamma_{j^\prime}^{(k-1)})}(-1)^{j-1}\left(n_{j^\prime}^{(k-1)}\right)_j\\
&=\sum_{j=1}^{2p\sigma(\gamma_1^{(k-1)})}(-1)^{j-1}m_j+
\sum_{j^\prime=2}^{P}(-1)^{j^\prime-1}\sum_{j=1}^{2p\sigma(\gamma_{j^\prime}^{(k-1)})}(-1)^{j-1}m_{\sum_{j_0=1}^{j^\prime-1}2p\sigma(\gamma_{j_0}^{(k-1)})+j}\\
&=\sum_{j=1}^{2p\sigma(\gamma_1^{(k-1)})}(-1)^{j-1}m_j+
\sum_{j^\prime=2}^{P}(-1)^{j^\prime-1}
\sum_{j=\sum_{j_0=1}^{j^\prime-1}2p\sigma(\gamma_{j_0}^{(k-1)})+1}^{\sum_{j_0=1}^{j^\prime}2p\sigma(\gamma_{j_0}^{(k-1)})}(-1)^{j-1-\sum_{j_0=1}^{j^\prime-1}2p\sigma(\gamma_{j_0}^{(k-1)})}m_{j}\\
&=\sum_{j=1}^{2p\sigma(\gamma_1^{(k-1)})}(-1)^{j-1}m_j+
\sum_{j^\prime=2}^{P}(-1)^{j^\prime-1}
\sum_{j=\sum_{j_0=1}^{j^\prime-1}2p\sigma(\gamma_{j_0}^{(k-1)})+1}^{\sum_{j_0=1}^{j^\prime}2p\sigma(\gamma_{j_0}^{(k-1)})}(-1)^{j-1}\cdot(-1)^{j^\prime-1}m_{j}\\
&=\sum_{j=1}^{2p\sigma(\gamma_1^{(k-1)})}(-1)^{j-1}m_j+
\sum_{j^\prime=2}^{P}
\sum_{j=\sum_{j_0=1}^{j^\prime-1}2p\sigma(\gamma_{j_0}^{(k-1)})+1}^{\sum_{j_0=1}^{j^\prime}2p\sigma(\gamma_{j_0}^{(k-1)})}(-1)^{j-1}m_{j}\\
&=\sum_{j=1}^{2p\sigma(\gamma^{(k)})}(-1)^{j-1}m_j.
\end{align*}
%Here we use the fact that $2p\sigma$ is odd for the third equal sign from back to front; see \autoref{so} in Appendix.
This proves that \eqref{as2} is true for $k$.

It follows from induction that \eqref{as2} holds for all $k\geq1$. This completes the proof of Lemma \autoref{lemas}.
\end{proof}

\begin{rema}
The above proof is analytical, which seems complicated. Below we will introduce a Feynman diagram to better understand it.
\end{rema}
%with the combinatorial lattice space $\cN$ and the alternating sums$\cc\ca\cs$, the alternating discrete convolution of higher dimensions reads
%\begin{align}\label{adc2}
%f_1\star\cdots\star f_P(n):=\sum_{\substack{n^{(\cdot)}\in\cN\\\cc\ca\cs(n(\cdot))=n}}
%\prod_{j=1}^P\{f(n_j)\}^{\ast^{[j-1]}}.
%\end{align}

\subsection{Notations of $\ll$ and $|\cdot|$ }{Let ``$\ll$" be ``$\leq$" in the sense of $|\cdot|$, that is, $Q_1\ll Q_2$ means that $|Q_1|\leq Q_2$.

Throughout this paper, we will abuse the symbol of $|\cdot|${, without causing any confusion, to stand for the modulus of a complex number, the absolute value of a real number, the $\ell^1$-norm of a vector, or the length of a multi-index}. That is,
\begin{itemize}
  \item $|z|=\sqrt{z\overline{z}}$, where $z\in\bC$;
  \item $|n\cdot\omega|$ stands for the absolute value of $n\cdot\omega\in\bR$;
  \item $|n|=\sum_{j=1}^{\nu}|n_j|$ and $|\omega|=\sum_{j=1}^{\nu}|\omega_j|$, where $n=(n_j)_{1\leq j\leq\nu}\in\bZ^\nu$ and $\omega=(\omega_j)_{1\leq j\leq\nu}\in\bR^\nu$;
      \item $|\alpha|=\sum_{j=1}^{r}\alpha_j$, where $\alpha=(\alpha_1,\cdots,\alpha_r)\in\bN^r$.
\end{itemize}
As to $|\cdot|$, we have
\begin{align}\label{ine1}
n\cdot\omega\ll|n||\omega|
\end{align}
%(see \autoref{simp} in Appendix)
and the triangle inequality:
\begin{align}\label{ine2}
v_1+v_2\ll|v_1|+|v_2|,\quad\forall~v_1,v_2\in\bR^\nu.
\end{align}
%(see \autoref{l1t} in Appendix)
}
\section{cNLS}\label{seccnls}
In this section, we consider the following nonlinear Schr\"odinger equation with higher order algebraic power-type nonlinearity
\eqref{cnls}
with quasi-periodic initial data \eqref{id0}
on the real line $\bR$, where $p\in\bN$ and $\lambda=\pm1$ denotes the focusing case $(\lambda=+1)$ and defocusing case ($\lambda=-1$) respectively,
%\footnote{The associated hamiltonian is $H=\int|\partial_xu|^2/2-\lambda\int|u|^{2p+2}/2p+2$, which is elliptic in the defocusing case and hyperbolic in the focusing case.}
$\{c(n)\}_{n\in\mathbb Z}$ is a sequence of the initial Fourier data, and it is $\mathtt r$-polynomially decaying, or rather, there exists a pair of constants $(A,\mathtt r)\in(0,\infty)\times(0,\infty)$ such that
\begin{align}\label{pd}
c(n)\ll A^{\frac{1}{2p}}(1+|n|)^{-\mathtt r},\quad\forall n\in\bZ^\nu.
\end{align}

%For the sake of convenience, we call \eqref{cnls}-\eqref{id} {\bf the quasi-periodic Cauchy problem for the cNLS}. What we are interested in is to study {\bf the spatially quasi-periodic solutions with the same frequency vector as the initial data} to this problem. Such a solution is defined by the following Fourier series
%\begin{align}\label{ses}
%u(t,x)=\sum_{n\in\bZ^\nu}c(t,n) e^{{\rm i}\langle n\rangle x},
%\end{align}
%where $c(t,n)$ is the unknown for all $n\in\bZ^\nu$ on a suitable time interval.

%\marginpar{We have changed ``folowing" into ``following".}
{For the quasi-periodic Cauchy problem \eqref{cnls}-\eqref{id0}, our main result is the following Theorem \ref{cnlsth}.}
\begin{theo}[cNLS]\label{cnlsth}
If $2\leq\nu<\min\left\{\frac{\mathtt r}{2}-2,\frac{\mathtt r}{4}\right\}=\frac{\mathtt r}{4}$, where $\mathtt r>8$, and the initial Fourier data $c$ is $\mathtt r$-polynomially decaying, i.e., it satisfies the polynomial decay condition \eqref{pd}, then
\begin{enumerate}
  \item (Existence)~the quasi-periodic Cauchy problem \eqref{cnls}-\eqref{id0} has a spatially quasi-periodic solution \eqref{ses} with the same frequency vector as the initial data (i.e., it {retains} the same spatial quasi-periodicity) defined on $[0,t_0]\times\bR$, where $t_0=A^{-1}\left(\mathfrak b(\mathtt r/2;\nu)\right)^{-2p}(2p)^{2p}P^{-P}$ (see \eqref{t0});
  \item (Decay and smoothness)~the spatially quasi-periodic solution \eqref{ses} is, {uniformly in $t$}, $\mathtt r/2$-polynomially decaying (with a slight worse decay rate), that is, \begin{align}\label{sdecay}
      |c(t,n)|\lesssim (1+|n|)^{-\mathtt r/2},\quad \forall (t,n)\in[0,t_0]\times\mathbb Z^\nu.
       \end{align}
       Hence this solution is in the classical sense;
  \item (Uniqueness)~the spatially quasi-periodic solution \eqref{ses} with polynomially decaying Fourier coefficients \eqref{sdecay} is unique on $[0,t_0]\times\bR$;
  \item (Up to the assigned time horizon)~for any given $T>0$, if the decay rate $\mathtt r$, the amplitude $A$, the dimension $\nu$ and the degree of nonlinearity $p$ satisfy \[A^{-1}\left(\mathfrak b(\mathtt r/2;\nu)\right)^{-2p}(2p)^{2p}P^{-P}\geq T,\]
then the unique spatially quasi-periodic solution is well-defined on $[0,T]\times\bR$.
%\item (Asymptotic dynamics)~Consider the quasi-periodic Cauchy problem for the standard NLS \eqref{cnls} with small nonlinearity depicted by a small parameter $\epsilon: 0<|\epsilon|\ll1$, that is,
%\begin{align}
%\tag{$\epsilon$-cNLS}{\rm i}\partial_tu+\partial_{x}^2u+\epsilon|u|^{2p}u=0,\quad 2\leq p\in\mathbb N.
%\end{align}
%Then for $t=|\epsilon|^{-1+\eta}\leq|\epsilon|^{-1}\sim T_\epsilon$ with $0<\eta\ll1$, we have
%\[|u(t,x)-\sum_{n\in\mathbb Z^\nu}e^{-{\rm i}\langle n\rangle^2 t}c(n)e^{{\rm i}\langle n\rangle x}|_{L_x^\infty(\mathbb R)}\rightarrow0,\quad{\text{as}}~\epsilon\rightarrow0.\]
\item (Asymptotic dynamics)~Consider the quasi-periodic Cauchy problem for the standard NLS \eqref{cnls} with small nonlinearity depicted by a small parameter $\epsilon: 0<|\epsilon|\ll1$, that is,
\begin{align}\label{ecnls}
\tag{$\epsilon$-cNLS}{\rm i}\partial_tu+\partial_{x}^2u+\epsilon|u|^{2p}u=0,\quad 2\leq p\in\mathbb N.
\end{align}
Let $u^\epsilon$ and $u_{\text{linear}}$  respectively be nonlinear and linear solutions to \eqref{ecnls}.
Then for $t=|\epsilon|^{-1+\eta}\leq|\epsilon|^{-1}\sim T_\epsilon$ with $0<\eta\ll1$, as $\epsilon\rightarrow0$, we have
\begin{align}
L^\infty\text{-asymptoticity}:\quad&\|u^\epsilon(t)-u_{\text{linear}}(t)\|_{L_x^\infty(\mathbb R)}\rightarrow 0;\\
\text{Sobolev asymptoticity}:\quad&
%\text{dist}~(u^{\epsilon},u_{\text{linear}})=
\|u^\epsilon(t)-u_{\text{linear}}(t)\|_{H_x^{s}(\mathbb R)}\rightarrow 0, \quad(s<\mathtt r/4-\nu/2),
\end{align}
where $\|f\|_{L_x^{\infty}(\mathbb R)}=\max_{x\in\mathbb R}|f(x)|$ and $\|f\|_{H_x^s(\mathbb R)}%:&=\left\{\sum_{m=0}^{s}\|\partial_x^mf\|_{L^2(\mathbb R)}\right\}^{1/2}\\
=\left\{\sum_{m=0}^s\|\langle n\rangle^m\hat{f}(\langle n\rangle)\|_{\ell_{n}^2(\mathbb Z^\nu)}\right\}^{1/2}$.
\end{enumerate}
\end{theo}
%\marginpar{(TBD)}

%\begin{rema}
%On the one hand, for $\nu=1$, that is, the spatially periodic case, the decay rate $\mathtt r$ of the initial Fourier data need to be only in $(6,8]$
%%\marginpar{Should the case $r = 6$ be excluded?\\Yes, $r=6$ is not included.}
%and the associated restriction condition is $1=\nu<\frac{\mathtt r}{2}-2$ (if $\mathtt r>8$, then $1=\nu<\frac{\mathtt r}{4}$). On the other hand, for $\nu\geq2$, that is, the spatially quasi-periodic case, the decay rate $\mathtt r$ of the initial Fourier data should be greater than $8$.
%
%\end{rema}

{We will divide the proof of Theorem \ref{cnlsth} into the following subsections. }
\subsection{Infinite-Dimensional ODEs and Picard Iteration
%with discrete convolution of higher dimension for cNLS with Fourier series
}
In this subsection we give an equivalent description for the quasi-periodic Cauchy problem \eqref{cnls}-\eqref{id0} in the Fourier space.

First we expand the nonlinearity $|u|^{2p}u$ in terms of $e^{{\rm i}\langle n\rangle x}$, that is,
%we calculate its Fourier coefficient{s}.
%For the sake of simplicity, we first take $p=1$ as an example. It follows from $|z|^2=z\bar z$ for any $z\in\bC$ and \eqref{c} that
%\begin{align*}
%|u^2|u&=u\bar uu\\
%&=\sum_{n_1\in\bZ^\nu}c(t,n_1)e^{{\rm i}\langle n_1\rangle x}\cdot\sum_{{n_2}\in\bZ^\nu}\overline{c(t,{n_2})}e^{-{\rm i}\langle n_2\rangle x}\cdot\sum_{{n_3}\in\bZ^\nu}c(t,n_3)e^{{\rm i}\langle n_3\rangle x}\\
%&=\sum_{n_1,n_2,n_3\in\bZ^\nu}c(t,n_1)\overline{c(t,n_2)}c(t,n_3)e^{{\rm i}\langle n_1-n_2+n_3\rangle x}\\
%&=\sum_{n\in\bZ^\nu}\sum_{\substack{n_1,n_2,n_3\in\bZ^\nu\\n_1-n_2+n_3=n}}
%c(t,n_1)\overline{c(t,n_2)}c(t,n_3)e^{{\rm i}\langle n\rangle x}\\
%&=\sum_{n\in\bZ^\nu}\sum_{\substack{n_j\in\bZ^\nu,~~j=1,\cdots,3\\\sum_{j=1}^3(-1)^{j-1}n_j=n}}
%\prod_{j=1}^3\left\{c(t,n_j)\right\}^{\ast^{[j-1]}}e^{{\rm i}\langle n\rangle x}.
%\end{align*}
%
%Similarly, for the general case, that is, $|u|^{2p}u$,  one can derive that
\begin{align}\label{oc}
|u|^{2p}u
=\sum_{n\in\bZ^\nu}\sum_{\substack{n_j\in\bZ^\nu,~~j=1,\cdots,P\\\sum_{j=1}^{P}(-1)^{j-1}n_j=n}}
\prod_{j=1}^{P}\left\{c(t,n_j)\right\}^{\ast^{[j-1]}}e^{{\rm i}\langle n\rangle x}.
\end{align}
It should be emphasized again that this is the first time we use the power of $\ast^{\cdot}$ to label the complex conjugate; see \autoref{lc}. Note that above result appears in the form of $\ast^{[\cdot]}$ in order to get a unified result in a manner of combinatorics. Although it seems ``simple" here, it will play an essential role in the Picard iteration; see Lemma \autoref{lee}.

Since our method works for both the focusing and the defocusing NLS, we consider only the former case, that is, $\lambda=+1$. Formally (``$\partial\sum=\sum\partial$"), we know that the quasi-periodic Cauchy problem \eqref{cnls}-\eqref{id0} is equivalent to the following nonlinear system of infinite coupled ODEs
\begin{align}\label{ode}
(\partial_tc)(t,n)+{\rm i}\langle n\rangle c(t,n)={\rm i}\sum_{\substack{n_j\in\bZ^\nu,~~j=1,\cdots,P\\\sum_{j=1}^{P}(-1)^{j-1}n_j=n}}
\prod_{j=1}^{P}\left\{c(t,n_j)\right\}^{\ast^{[j-1]}},
\end{align}
which has the following integral form
\begin{align}\label{ie}
c(t,n)=e^{-{\rm i}\langle n\rangle^2t}c(n)+{\rm i}\int_0^te^{-{\rm i}\langle n\rangle^2(t-s)}\sum_{\substack{n_j\in\bZ^\nu,~~j=1,\cdots,P\\\sum_{j=1}^{P}(-1)^{j-1}n_j=n}}
\prod_{j=1}^{P}\left\{c(s,n_j)\right\}^{\ast^{[j-1]}}{\rm d}s.
\end{align}

According to the feedback of nonlinearity, define the Picard iteration $\{c_k(t,n)\}_{k\in\mathbb N}$ to uniformly approximate $c(t,n)$ as follows: first we choose the solution $$c_0(t,n)=e^{-{\rm i}\langle n\rangle^2t}c(n)
%\quad(\text{initial guess})
$$
to the linear equation of \eqref{cnls} as the initial guess and then define $c_k(t,n)$ successively by letting
\begin{align}\label{ck}
c_k(t,n)=c_0(t,n)+{\rm i}\int_0^te^{-{\rm i}\langle n\rangle^2(t-s)}\sum_{\substack{n_j\in\bZ^\nu,~~j=1,\cdots,P\\\sum_{j=1}^{P}(-1)^{j-1}n_j=n}}
\prod_{j=1}^{P}\left\{c_{k-1}(s,n_j)\right\}^{\ast^{[j-1]}}{\rm d}s,\quad \forall k\geq1.
\end{align}

\begin{rema}
This iteration is complicated. In fact, let $N_k$ be the term number of the initial Fourier data, where $k\geq1$. Then $N_1=2$ and $N_k=1+N_{k-1}^{P}$ for all $k\geq2$. Exponential growth and alternating complex conjugate force us to control $c_k$ well. To overcome these difficulties, we use an explicit combinatorial method with pcc and Feynman diagram to analyze the Picard iteration.
\end{rema}
\subsection{Combinatorial Tree}In this subsection we give {a} combinatorial tree form of $c_k(t,n)$ with the help of $\ast^{[\cdot]}$; see \autoref{lc}), that is, the following Lemma \ref{lee}.
\begin{lemm}\label{lee}
For all $k\geq1$ and $n\in\bZ^\nu$,  $c_k(t,n)$ has the following combinatorial tree form
\begin{align}\label{ct}
c_k(t,n)=\sum_{\gamma^{(k)}\in\Gamma^{(k)}}\sum_{\substack{n^{(k)}\in\cN^{(k,\gamma^{(k)})}\\\cc\ca\cs(n^{(k)})=n}}
\cC^{(k,\gamma^{(k)})}(n^{(k)})\cI^{(k,\gamma^{(k)})}(t,n^{(k)})\cF^{(k,\gamma^{(k)})}(n^{(k)}){,}
\end{align}
where $\cC,\cI$ and $\cF$ {(here we omit the index $(k,\gamma^{(k)})$ in these abstract symbols for simplicity, which will be used below somewhere without causing any confusions)} are iteratively defined as follows: For $k\geq1,\gamma^{(k)}=0, n^{(k)}\in\cN^{(k,0)}$, set
\begin{align*}
\cC^{(k,0)}(n^{(k)})&=c\left(\cc\ca\cs(n^{(k)})\right),\\
\cI^{(k,0)}(t,n^{(k)})&=e^{-{\rm i}\left\langle\cc\cc\ca\cs(n^{(k)})\right\rangle^2t},\\
\cF^{(k,0)}(n^{(k)})&=1;
\end{align*}
For $k=1,\gamma^{(1)}=1\in\Gamma^{(1)},n^{(1)}=(n_j)_{1\leq j\leq P}\in\cN^{(1,1)}$, set
\begin{align*}
\cC^{(1,1)}(n^{(1)})&=\prod_{j=1}^{P}\left\{c(n_j)\right\}^{\ast^{[j-1]}},\\
\cI^{(1,1)}(t,n^{(1)})&=\int_0^te^{-{\rm i}\left\langle\cc\ca\cs(n^{(1)})\right\rangle^2(t-s)}\prod_{j=1}^{P}\left\{e^{-{\rm i}\langle n_j\rangle^2s}\right\}^{\ast^{[j-1]}}{\rm d}s,\\
\cF^{(1,1)}(n^{(1)})&={\rm i};
\end{align*}
For $k\geq2, \gamma^{(k)}=(\gamma_j^{(k-1)})_{1\leq j\leq P}\in(\Gamma^{(k-1)})^{P},
n^{(k)}=(n_j^{(k-1)})_{1\leq j\leq P}$, set
\begin{align*}
\cC^{(k,\gamma^{(k)})}(n^{(k)})&=\prod_{j=1}^{P}\left\{\cC^{(k-1,\gamma_j^{(k-1)})}(n^{(k-1)}_j)\right\}^{\ast^{[j-1]}},\\
\cI^{(k,\gamma^{(k)})}(t,n^{(k)})&=\int_0^te^{-{\rm i}\left\langle\cc\ca\cs(n^{(k)})\right\rangle^2(t-s)}\prod_{j=1}^{P}\left\{\cI^{(k-1,\gamma_j^{(k-1)})}(s,n_j^{(k-1)})\right\}^{\ast^{[j-1]}}{\rm d}s,\\
\cF^{(k,\gamma^{(k)})}(n^{(1)})&={\rm i}\prod_{j=1}^{P}\left\{\cF^{(k-1,\gamma_j^{(k-1)})}(n_j^{(k-1)})\right\}^{\ast^{[j-1]}}.
\end{align*}
\end{lemm}
\begin{proof}
For the initial guess/the linear part of \eqref{ck}, we have
\begin{align*}
c_0(t,n)%&=e^{-{\rm i}\langle n\rangle^2t}c(n)\\
&=\sum_{\substack{n^{(k)}\in\cN^{(k,0)}\\\cc\ca\cs(n^{(k)})=n}}
\cC^{(k,0)}(n^{(k)})\cI^{(k,0)}(t,n^{(k)})\cF^{(k,0)}(n^{(k)})\mapsto\gamma^{(k)}=0\in\Gamma^{(k)},\quad\forall k\geq1.
\end{align*}
{Here we use $``\mapsto"$ to show that every term in the Picard sequence is labeled by a unique element in the branch set, which will be used below as well.}

For $k=1$ and the nonlinear part of \eqref{ck}, one can derive that
\begin{align*}
c_1(t,n)-c_0(t,n)%&={\rm i}\int_0^te^{-{\rm i}\langle n\rangle^2(t-s)}\sum_{\substack{n_j\in\bZ^\nu,~~j=1,\cdots,P\\\sum_{j=1}^{P}(-1)^{j-1}n_j=n}}
%\prod_{j=1}^{P}\left\{c_{0}(s,n_j)\right\}^{\ast^{[j-1]}}{\rm d}s\\
&=\sum_{\substack{n_j\in\bZ^\nu,~~j=1,\cdots,P\\\sum_{j=1}^{P}(-1)^{j-1}n_j=n}}
\prod_{j=1}^{P}\left\{c(n_j)\right\}^{\ast^{[j-1]}}\cdot\int_0^te^{-{\rm i}\langle n\rangle^2(t-s)}\prod_{j=1}^{P}\left\{e^{-{\rm i}\langle n_j\rangle^2s}\right\}^{\ast^{[j-1]}}{\rm d}s\cdot{\rm i}\\
&=\sum_{\substack{n^{(1)}\in\cN^{(1,1)}\\\cc\ca\cs(n^{(1)})=n}}\cC^{(1,1)}(n^{(1)})\cI^{(1,1)}
(t,n^{(1)})\cF^{(1,1)}(n^{(1)})\mapsto\gamma^{(1)}=1\in\Gamma^{(1)}.
\end{align*}
Thus
\begin{align*}
c_1(t,n)%&=(c_1(t,n)-c_0(t,n))+c_0(t,n)\\
&=\sum_{\gamma^{(1)}\in\Gamma^{(1)}}\sum_{\substack{n^{(1)}\in\cN^{(1,\gamma^{(1)})}
\\\cc\ca\cs(n^{(1)})=n}}\cC^{(1,\gamma^{(1)})}(n^{(1)})\cI^{(1,\gamma^{(1)})}
(t,n^{(1)})\cF^{(1,\gamma^{(1)})}(n^{(1)}).
\end{align*}
This implies that \eqref{ct} is true for $k=1$.

Let $k\geq2$. Assume that \eqref{ct} holds for all $1<k^\prime<k$.

For $k$, it follows from the definition of $c_k$, i.e., \eqref{ck}, and induction that the nonlinear part of \eqref{ck} reads
\begin{align*}
&c_k(t,n)-c_0(t,n)\\
%=&{\rm i}\int_0^te^{-{\rm i}\langle n\rangle^2(t-s)}\sum_{\substack{n_j\in\bZ^\nu,~~j=1,\cdots,P\\\sum_{j=1}^{P}(-1)^{j-1}n_j=n}}
%\prod_{j=1}^{P}\left\{c_{k-1}(s,n_j)\right\}^{\ast^{[j-1]}}{\rm d}s\\
=&{\rm i}\int_0^te^{-{\rm i}\langle n\rangle^2(t-s)}\sum_{\substack{n_j\in\bZ^\nu,~~j=1,\cdots,P\\\sum_{j=1}^{P}(-1)^{j-1}n_j=n}}
\prod_{j=1}^{P}\\
&\Bigg\{\sum_{\gamma_j^{(k-1)}\in\Gamma^{(k-1)}}\sum_{\substack{n_j^{(k-1)}\in\cN^{(k-1,\gamma_j^{(k-1)})}
\\\cc\ca\cs(n_j^{(k-1)})=n_j}}\cC^{(k-1,\gamma_j^{(k-1)})}(n_j^{(k-1)})\cI^{(k-1,\gamma_j^{(k-1)})}
(s,n_j^{(k-1)})\\
&\cF^{(k-1,\gamma_j^{(k-1)})}(n_j^{(k-1)})\Bigg\}^{\ast^{[j-1]}}{\rm d}s\\
=&{{\rm i}\int_0^te^{-{\rm i}\langle n\rangle^2(t-s)}\sum_{\substack{n_j\in\bZ^\nu,~~j=1,\cdots,P\\\sum_{j=1}^{P}(-1)^{j-1}n_j=n}}
\prod_{j=1}^{P}\sum_{\gamma_j^{(k-1)}\in\Gamma^{(k-1)}}\sum_{\substack{n_j^{(k-1)}\in\cN^{(k-1,\gamma_j^{(k-1)})}
\\\cc\ca\cs(n_j^{(k-1)})=n_j}}}\\
&{\left\{\cC^{(k-1,\gamma_j^{(k-1)})}(n_j^{(k-1)})\right\}^{\ast^{[j-1]}}\left\{\cI^{(k-1,\gamma_j^{(k-1)})}
(s,n_j^{(k-1)})\right\}^{\ast^{[j-1]}}\left\{\cF^{(k-1,\gamma_j^{(k-1)})}(n_j^{(k-1)})\right\}^{\ast^{[j-1]}}{\rm d}s}\\
=&\sum_{\substack{\gamma_j^{(k-1)}\in\Gamma^{(k-1)}\\j=1,\cdots,P}}
\sum_{\substack{n_j\in\bZ^\nu,~~j=1,\cdots,P\\\sum_{j=1}^{P}(-1)^{j-1}n_j=n}}\sum_{\substack{n_j^{(k-1)}\in\cN^{(k-1,\gamma_j^{(k-1)})}
\\\cc\ca\cs(n_j^{(k-1)})=n_j}}\prod_{j=1}^{P}\left\{\cC^{(k-1,\gamma_j^{(k-1)})}(n_j^{(k-1)})\right\}^{\ast^{[j-1]}}
\\
&\cdot\int_0^te^{-{\rm i}\langle n\rangle^2(t-s)}\prod_{j=1}^{P}\left\{\cI^{(k-1,\gamma_j^{(k-1)})}(s,n_j^{(k-1)})\right\}^{\ast^{[j-1]}}
{\rm d}s\cdot{\rm i}\prod_{j=1}^{P}\left\{\cF^{(k-1,\gamma_j^{(k-1)})}(n_j^{(k-1)})\right\}^{\ast^{[j-1]}}\\
=&\sum_{\substack{\gamma^{(k)}=(\gamma_j^{(k-1)})_{1\leq j\leq P}\\\gamma_{j}^{(k-1)}\in\Gamma^{(k-1)}\\j=1,\cdots,P}}
\sum_{\substack{n^{(k)}=(n_j^{(k-1)})_{1\leq j\leq P}\\
n_j^{(k-1)}\in\cN^{(k-1,\gamma_j^{(k-1)})}
\\\cc\ca\cs(n_j^{(k-1)})=n_j\\j=1,\cdots,P}}\prod_{j=1}^{P}\left\{\cC^{(k-1,\gamma_j^{(k-1)})}(n_j^{(k-1)})\right\}^{\ast^{[j-1]}}
\\
&\cdot\int_0^te^{-{\rm i}\langle n\rangle^2(t-s)}\prod_{j=1}^{P}\left\{\cI^{(k-1,\gamma_j^{(k-1)})}(s,n_j^{(k-1)})\right\}^{\ast^{[j-1]}}
{\rm d}s\cdot{\rm i}\prod_{j=1}^{P}\left\{\cF^{(k-1,\gamma_j^{(k-1)})}(n_j^{(k-1)})\right\}^{\ast^{[j-1]}}\\
=&\sum_{\gamma^{(k)}\in(\Gamma^{(k-1)})^{P}}\sum_{\substack{n^{(k)}\in\cN^{(k,\gamma^{(k)})}\\
\cc\ca\cs(n^{(k)})=n}}\cC^{(k,\gamma^{(k)})}(n^{(k)})\cI^{(k,\gamma^{(k)})}(t,n^{(k)})\cF^{(k,\gamma^{(k)})}(n^{(k)})
\mapsto\gamma^{(k)}\in(\Gamma^{(k-1)})^{P}.
\end{align*}
As a result, we have
\begin{align*}
c_k(t,n)%&=(c_k(t,n)-c_0(t,n))+c_0(t,n)\\
&=\sum_{\gamma^{(k)}\in\Gamma^{(k)}}\sum_{\substack{n^{(k)}\in\cN^{(k,\gamma^{(k)})}
\\\cc\ca\cs(n^{(k)})=n}}\cC^{(k,\gamma^{(k)})}(n^{(k)})\cI^{(k,\gamma^{(k)})}
(t,n^{(k)})\cF^{(k,\gamma^{(k)})}(n^{(k)}).
\end{align*}
Hence \eqref{ct} is true for $k$.

It follows from induction that \eqref{ct} holds for all $k\in\bN$. This completes the proof of Lemma \ref{lee}.
\end{proof}

\subsection{Feynman Diagram}\label{di}
%see \autoref{p2} in Appendix.

%\marginpar{We add a diagram interpretation of combinatorial tree. }

In this subsection, we take $p=1$ as a typical case to introduce the Feynman diagram to illustrate the combinatorial tree \eqref{ct}.

% for these combinatorial concepts appearing in , which  is defined by the iteration pattern \eqref{ck}.

What \eqref{ct} means is that every term $c_k$ can be imagined as a tree. Such a tree is presented by a Feynman diagram which is named after $\Gamma^{(k)}$-family, with the help of the branch set $\Gamma^{(k)}$, for each $k\geq1$.

For any given $\Gamma^{(k)}$, a branch of this tree is circled by a rectangular box and it is named after an element $\gamma^{(k)}\in\Gamma^{(k)}$.

Let black/red/green lattice points $\bullet/{\color{red}\bullet}/{\color{green}\bullet}$ be the elements in the combinatorial lattice $\cN^{(k,\gamma^{(k)})}$. They are equipped with positive and negative signs, or rather, black/red lattice points $\bullet/{\color{red}\bullet}$ have positive sign $+$, and green lattice points ${\color{green}\bullet}$ have negative sign $-$. What's more, from the perspective of a multi-linear operator, black/red lattice points $\bullet/{\color{red}\bullet}$ as independent variables correspons to the initial data $c$, and green lattice points ${\color{green}\bullet}$ as independent variables correspond to the complex conjugate of the initial data $c$, i.e., $\overline{c}$.

For $\Gamma^{(k)}$, there are $k+1$ horizontal lines which are named after level $j$, where $j=0,1,\cdots,k$, from up to down.
On level $0$, there is exactly a black point $\bullet$ labelled by $n$. On level $k$, there are exactly $2\sigma(\gamma^{(k)})$  points $m_j$ with $j=1,\cdots,2\sigma(\gamma^{(k)})$, and they are ordered from left to right. If $\gamma^{(k)}=0$, $m_1$ is presented as a black point $\bullet$; if $\gamma^{(k)}\in\Gamma^{(k)}\backslash\{0\}$, $m_j$'s are presented as magenta points ${\color{red}\bullet}$ if $j$ is odd and green points ${\color{green}\bullet}$  if $j$ is even.

We also put $n^{(k)}$ on the left of level $k$.

With such a tree at hand, we can easily write the combinatorial alternating sum as follows: $$\cc\ca\cs(n^{(k)})=\sum_{j=1}^{2\sigma(\gamma^{(k)})}(-1)^{j-1}m_j.$$

Hence \eqref{ct} can be viewed as a sum over all branches of $\Gamma^{(k)}$-family, and each branch splits over the condition $\cc\ca\cs(n^{(k)})=n$.

The diagrams of $\Gamma^{(1)}$-family and $\Gamma^{(2)}$-family are given below; see \autoref{g1} and \autoref{g2} respectively.
%\marginpar{When compiling, these jpg files were reported as not available, so I have commented out these figures for now.\\
%FX: I will upload these jpg files in the shared folder.}

\begin{figure}[htbp]
  \centering
  \includegraphics[width=0.6\textwidth]{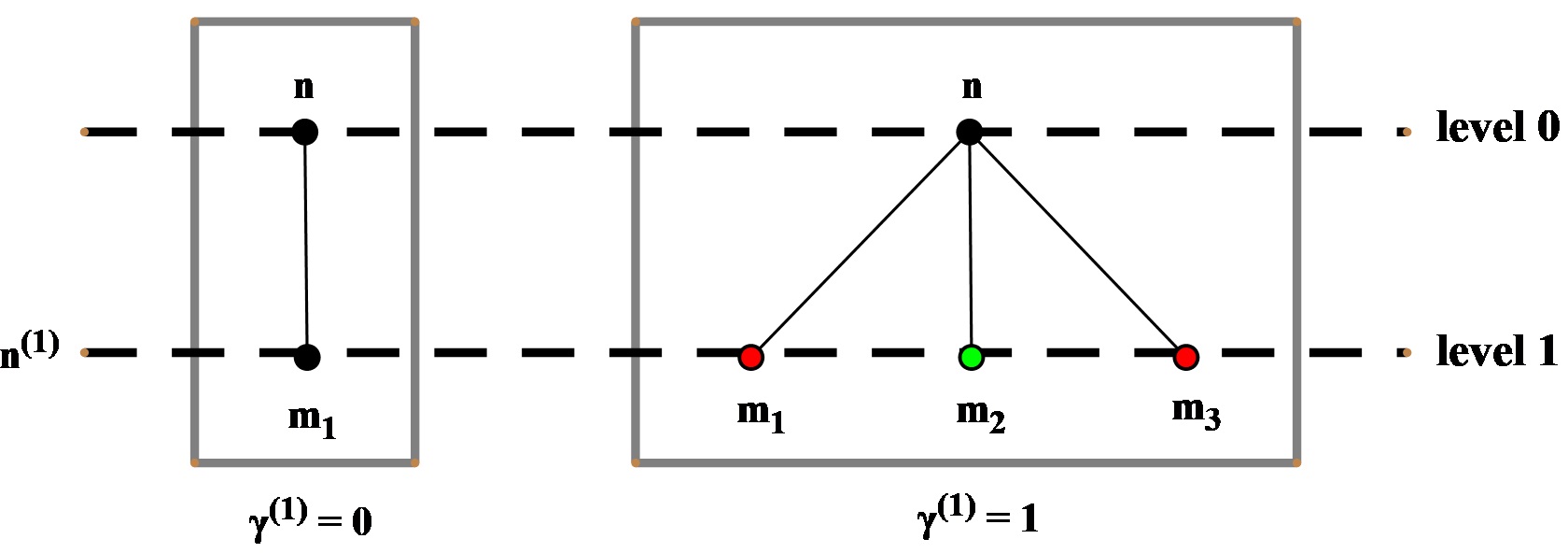}
  \caption{$\Gamma^{(1)}$-family.}
  \label{g1}
\end{figure}
(Description of \autoref{g1})~---~This diagram stands for the first iteration, that is, $c_1$ can be viewed as a sum of these two rectangular in the following sense: the left one depicts the linear part and the right one depicts the nonlinear part.

---~More precisely, in each rectangle, there are two horizontal dotted lines, named after level $0$ and level $1$, from up to down.

---~In level $0$, there is exactly one black point $\bullet$, named after the initial lattice point $n$.

---~For the left rectangle, the initial lattice point has only one child, put in the next level $1$, labeled by a black point $\bullet$ and named after $m_1$. Hence $$\cc\ca\cs(n^{(1)})=m_1 \quad\text{and}\quad \cC^{(1,0)}(n^{(1)})=c(m_1).$$

---~For the right rectangle, the initial lattice point has three children, put in the next level $1$, labeled by ${\color{red}\bullet}$, ${\color{green}\bullet}$ and ${\color{red}\bullet}$, named after $m_1, m_2$ and $m_3$ respectively. In this case, {red} and {green} colors $({\color{red}\bullet};{\color{green}\bullet})$ respectively stand for $({\color{red}+};{\color{green}-})$ and $({\color{red}z};{\color{green}\bar{z}})$ for any complex number $z\in\mathbb C$. Thus $$\cc\ca\cs(n^{(1)})=m_1-m_2+m_3 \quad\text{and}\quad \cC^{(1,1)}(n^{1})=c(m_1)\overline{c(m_2)}c(m_3).$$
\begin{figure}[htbp]
  \centering
  \includegraphics[width=0.86\textwidth]{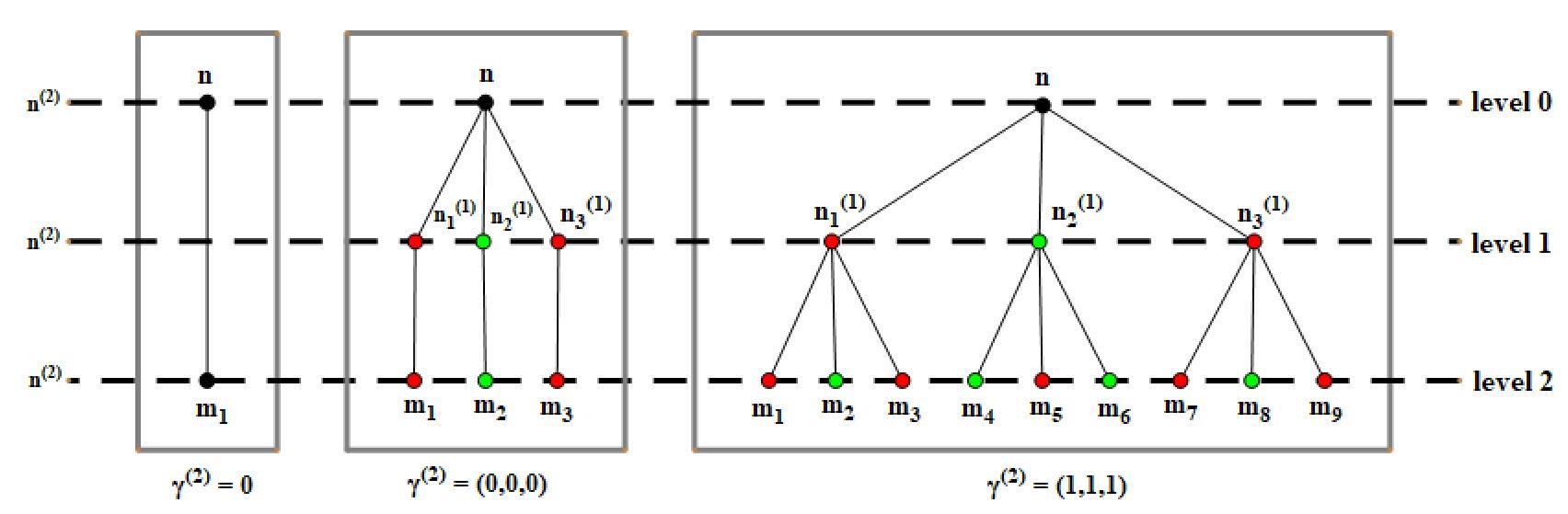}
  \includegraphics[width=0.9\textwidth]{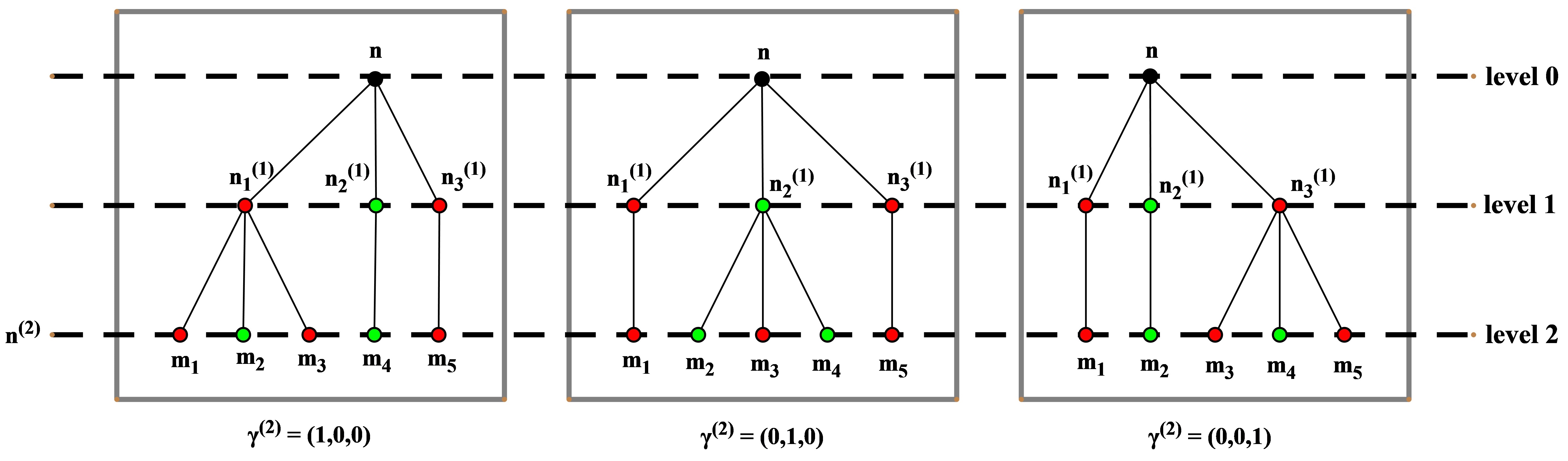}
  \includegraphics[width=0.9\textwidth]{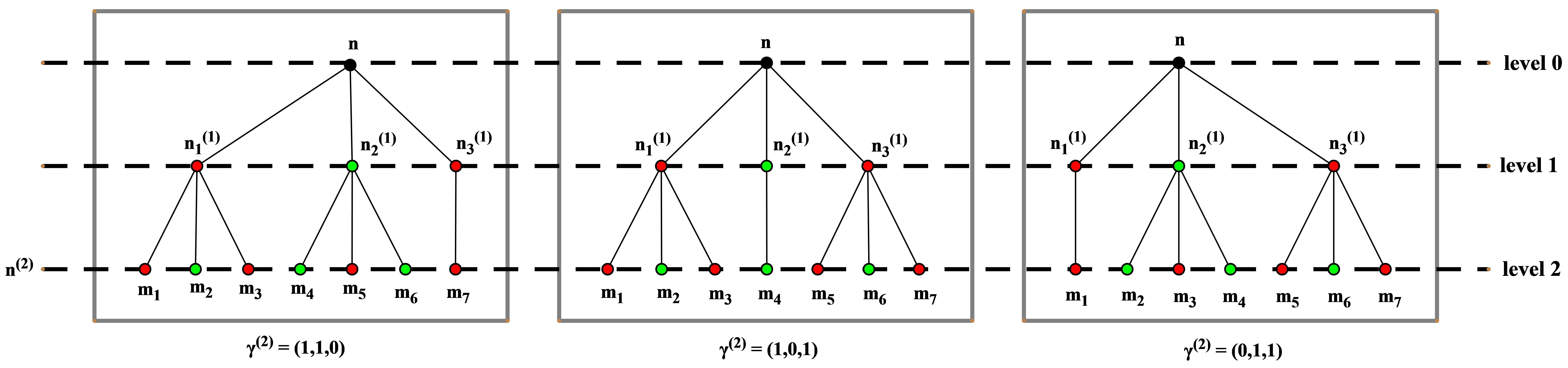}
  \caption{$\Gamma^{(2)}$-family.}\label{g2}
\end{figure}

(Description of \autoref{g2})~---~These nine rectangles depict the second iteration, that is, $c_2$ can be viewed as the summation of all these rectangles in the following sense: the first one depicts the linear part and the rest depict the nonlinear part.

---~In each rectangle there are three horizontal dotted lines, named after level $0,1,2$, from up to down.

---~In level $0$, there is exactly one black point $\bullet$, named after the initial lattice point $n$.

---~The first rectangle is the same as previous one in $c_1$ (notice that there is no point in the next level $1$, and a black point $\bullet$, put in the final level $2$, named after $m_1$).

---~For the remaining rectangles, the initial lattice points have exactly three children, put in the next level $1$, label by ${\color{red}\bullet}$, ${\color{green}\bullet}$ and ${\color{red}\bullet}$, named after $n_1^{(1)}, n_2^{(1)}$ and $n_3^{(1)}$ respectively (they consist of $n^{(2)}$).

---~In the second rectangle, each of the lattice point in level $1$ has exactly one child, put in the next level $2$, labeled by ${\color{red}\bullet}$, ${\color{green}\bullet}$ and ${\color{red}\bullet}$, named after $m_1, m_2$ and $m_3$ respectively. Thus $$\cc\ca\cs(n^{(2)})=m_1-m_2+m_3\quad\text{and}\quad \cC^{(2,(0,0,0))}(n^{(2)})=c(m_1)\overline{c(m_2)}c(m_3).$$ This corresponds to the nonlinearity $u\bar u u$ when every $u$ takes the linear part.

---~In the third rectangle, each of the lattice point in level $1$ has three children, put in the next level $2$, labeled by ${\color{red}\bullet}$, ${\color{green}\bullet}$ and ${\color{red}\bullet}$ or ${\color{green}\bullet}$, ${\color{red}\bullet}$ and ${\color{green}\bullet}$ (it is determined by the color of the first point, which is the same as its generator), named after $m_j$ with $j$ from $1$ to $9$ respectively (from left to right). In this case, $$\cc\ca\cs(n^{(2)})=m_1-m_2+m_3-m_4+m_5-m_6+m_7-m_8+m_9$$
and $$\cC^{(2,(1,1,1))}(n^{(2)})=c(m_1)\overline{c(m_2)}c(m_3)
\overline{c(m_4)}c(m_5)\overline{c(m_6)}c(m_7)
\overline{c(m_8)}c(m_9).$$ This corresponds to the nonlinearity $u\bar u u$ when every $u$ takes the nonlinear part.

---~In the second line, each of these diagrams, there is exactly one point in level $1$ that has three children and the rest have only one child. This corresponds to the the nonlinearity $u\bar u u$ when there is exactly one $u$ taking the nonlinear part and the rest taking the linear part.

---~In the third line, there are exactly two points in level $1$ respectively having three children and the rest has only one child. This corresponds to the nonlinearity $u\bar u u$ when there are exactly two $u$'s taking the nonlinear part and the rest taking the linear part.

%{\color{magenta}two points: the number of initial data and the domain of alternating discrete convolution of higher dimensions}

%\subsubsection{Diagram interpolation}\label{di}
%(TBD)

%\newpage
%\newpage
\subsection{The Picard Sequence is Polynomially Decaying With a Slightly Worse Decay Rate}

In this subsection we will prove that the Picard sequence is $\mathtt r/2$-polynomially decaying.

%with a slightly worse decay rate, or rather, from $\mathtt r$ to $\mathtt r/2$.

In order to complete this {estimate}, we first estimate the size of $\cC,\cI$ and $\cF$ independently; see {Lemmas} \ref{cl}, \ref{li} and \ref{fl} respectively.  Then putting these estimates together, we will be able {to estimate} the Picard sequence; see Lemma \ref{pl}.

We first estimate $\cC$. {To} this end, we need to {describe} $\cC^{(k,\gamma^{(k)})}(n^{(k)})$  {in terms of} the components $(n^{(k)})_j$ of $n^{(k)}$. Thanks to the label of complex conjugate, see \autoref{lc}, $\cC$ has a {\bf good {combinatorial}} structure w.r.t the initial Fourier data $c$, that is, the following Lemma \ref{ca}. Note that the result is different from \cite[Lemma 4.1]{DLX22AR}.
\begin{lemm}\label{ca}
For all $k\geq1$ and $n^{(k)}=(m_j)_{1\leq j\leq 2p\sigma(\gamma^{(k)})}$, we have
\begin{align}\label{cc}
\cC^{(k,\gamma^{(k)})}(n^{(k)})=\prod_{j=1}^{2p\sigma(\gamma^{(k)})}\left\{c(m_j)\right\}^{^{\ast^{[j-1]}}}.
\end{align}
\end{lemm}
\begin{proof}
For $k=1,\gamma^{(1)}=0,n^{(1)}=m_1\in\cN^{(1,0)}, 2p\sigma(0)=1$ and $k=1,\gamma^{(1)}=1,n^{(1)}=(m_j)_{1\leq j\leq P}\in\cN^{(1,1)}, 2p\sigma(1)=P$, it follows from the definition of $\cC$ that \eqref{cc} is true for $k=1$.

Let $k\geq2$. Assume that \eqref{cc} holds for all $1<k^\prime<k$.

For $k$, it's obvious that \eqref{cc} holds for $\gamma^{(k)}=0\in\Gamma^{(k)}$. For $\gamma^{(k)}=(\gamma_j^{(k-1)})_{1\leq j\leq P}\in(\Gamma^{(k-1)})^{P}$ and $n^{(k)}=(n_j^{(k-1)})_{1\leq j\leq p^{\ast}}\in\prod_{j=1}^{P}\cN^{(k-1,\gamma_j^{(k-1)})}$, where $$n_1^{(k-1)}=(m_{j})_{1\leq j\leq2p\sigma(\gamma_1^{(k-1)})}$$ and $$ n_{j^\prime}^{(k-1)}=(m_{j})_{2p\sum_{j=1}^{j^\prime-1}\sigma(\gamma_j^{(k-1)})+1\leq j\leq2p\sum_{j=1}^{j^\prime}\sigma(\gamma_j^{(k-1)})},\quad j^\prime=2,\cdots,P.$$ By the definition of $\cC$ and induction hypothesis, we have
\begin{align*}
&\cC^{(k,\gamma^{(k)})}(n^{(k)})\\
=~&
\prod_{j^\prime=1}^{P}
\left\{\cC^{(k-1,\gamma_{j^\prime}^{(k-1)})}(n_{j^\prime}^{(k-1)})\right\}^{\ast^{[j^\prime-1]}}\\
=~&\cC^{(k-1,\gamma_{1}^{(k-1)})}(n_{1}^{(k-1)})
\prod_{j^\prime=2}^{P}\left\{\cC^{(k-1,\gamma_{j^\prime}^{(k-1)})}(n_{j^\prime}^{(k-1)})\right\}^{\ast^{[j^\prime-1]}}\\
=~&\prod_{j=1}^{2p\sigma(\gamma_1^{(k-1)})}\left\{c(m_j)\right\}^{\ast^{[j-1]}}\prod_{j^\prime=2}^{P}
\left\{\prod_{j=1}^{2p\sigma(\gamma_{j^\prime}^{(k-1)})}
\left\{c(m_{2p\sum_{j_0=1}^{j^\prime-1}\sigma(\gamma_{j_0}^{(k-1)})+j})\right\}^{\ast^{[j-1]}}\right\}^{\ast
^{[j^\prime-1]}}\\
=~&\prod_{j=1}^{2p\sigma(\gamma_1^{(k-1)})}\left\{c(m_j)\right\}^{\ast^{[j-1]}}\prod_{j^\prime=2}^{P}
\left\{\prod_{j=2p\sum_{j_0=1}^{j^\prime-1}\sigma(\gamma_{j_0}^{(k-1)})+1}^{2p\sum_{j_0=1}^{j^\prime}
\sigma(\gamma_{j_0}^{(k-1)})}\{c(m_{j})\}^{\ast^{[j-2p\sum_{j_0=1}^{j^\prime-1}\sigma(\gamma_{j_0}^{(k-1)})-1]}}
\right\}^{\ast^{[j^\prime-1]}}\\
=~&\prod_{j=1}^{2p\sigma(\gamma_1^{(k-1)})}\left\{c(m_j)\right\}^{\ast^{[j-1]}}\prod_{j^\prime=2}^{P}
\prod_{j=2p\sum_{j_0=1}^{j^\prime-1}\sigma(\gamma_{j_0}^{(k-1)})+1}^{2p\sum_{j=1}^{j^\prime}
\sigma(\gamma_{j_0}^{(k-1)})}\left\{c(m_{j})\right\}^{\ast^{[j-1]}}\\
=~&\prod_{j=1}^{2p\sigma(\gamma^{(k)})}\left\{c(m_j)\right\}^{\ast^{[j-1]}}.
\end{align*}
Here we use the property of operation $\ast^{[\cdot]}$ (see \autoref{lc}) and Lemma \ref{lemms}. This shows that \eqref{cc} is true for $k$.

It follows from induction that \eqref{cc} holds for all $k\geq1$. This completes the proof of Lemma \ref{ca}.
\end{proof}

Thanks to Lemma \ref{ca}, we can directly estimate the size of $\cC$; see the following Lemma~\ref{cl}.
\begin{lemm}\label{cl}
If the initial Fourier data $c$ is $\mathtt r$-polynomially decaying, i.e., it satisfies \eqref{pd}, then
\begin{align}\label{cd}
|\cC^{(k,\gamma^{(k)})}(n^{(k)})|\leq A^{\sigma(\gamma^{(k)})}\prod_{j=1}^{2p\sigma(\gamma^{(k)})}\left(1+|(n^{(k)})_j|
\right)^{-\mathtt r},\quad \forall k\geq1.
\end{align}
\end{lemm}
\begin{proof}
For all $k\geq1, \gamma^{(k)}\in\Gamma^{(k)}$ and $n^{(k)}=(m_j)_{1\leq j\leq 2p\sigma(\gamma^{(k)})}$, it follows from Lemma \ref{ca} and \eqref{pd} that
\begin{align*}
|\cC^{(k,\gamma^{(k)})}(n^{(k)})|=&\prod_{j=1}^{2p\sigma(\gamma^{(k)})}|\left\{c(m_j)\right\}^{^{\ast^{[j-1]}}}|\\
%=&\prod_{j=1}^{2p\sigma(\gamma^{(k)})}\left|c(m_j)\right|\\
\leq&\prod_{j=1}^{2p\sigma(\gamma^{(k)})} A^{\frac{1}{2p}}(1+|m_j|)^{-\mathtt r}\\
=&A^{\sigma(\gamma^{(k)})}\prod_{j=1}^{2p\sigma(\gamma^{(k)})} \left(1+|(n^{(k)})_j|\right)^{-\mathtt r}.
\end{align*}
This completes the proof of Lemma \ref{cl}.
\end{proof}

Next we estimate $\cI$ and obtain the following Lemma \ref{li}.
\begin{lemm}\label{li}
For all $k\geq1$, we have
\begin{align}\label{dd}
|\cI^{(k,\gamma^{(k)})}(t,n^{(k)})|\leq\frac{t^{\ell(\gamma^{(k)})}}{\cD(\gamma^{(k)})},
\end{align}
where
%\begin{align}\label{sl}
%\ell(\gamma^{(k)})=\sigma(\gamma^{(k)})-\frac{1}{2p}
%\end{align}
%and
\begin{align}\label{pi}
\cD(\gamma^{(k)})=
\begin{cases}
1,&\gamma^{(k)}=0\in\Gamma^{(k)},k\geq1;\\
1,&\gamma^{(1)}=1\in\Gamma^{(1)};\\
\ell(\gamma^{(k)})\prod_{j=1}^{P}\cD(\gamma_j^{(k-1)}),&\gamma^{(k)}=(\gamma_j^{(k-1)})_{1\leq j\leq P}\in(\Gamma^{(k-1)})^{P},k\geq2.
\end{cases}
\end{align}
\end{lemm}
\begin{proof}
For all $k\geq1, \gamma^{(k)}=0\in\Gamma^{(k)}, \ell(0)=\sigma(0)-\frac{1}{2p}=0$ and $\cD(0)=1$ it follows from the definition of $\cI$ that
\begin{align*}
|\cI^{(k,0)}(t,n^{(k)})|%&=|e^{-{\rm i}\left\langle\cc\ca\cs(n^{(k)})\right\rangle^2t}|\\
%&=1\\
&=\frac{t^{\ell(0)}}{\cD(0)}.
\end{align*}

For $k=1, \gamma^{(1)}=1\in\Gamma^{(1)}, \ell(1)=\sigma(1)-\frac{1}{2p}=1$ and $\cD(1)=1$, by the definition of $\cI$, we have
\begin{align*}
|\cI^{(1,1)}(t,n^{(1)})|%=&|\int_0^te^{-{\rm i}\left\langle\cc\ca\cs(n^{(1)})\right\rangle^2(t-s)}\prod_{j=1}^{P}\left\{e^{-{\rm i}\langle n_j\rangle^2s}\right\}^{\ast^{[j-1]}}{\rm d}s|\\
\leq\int_0^t|e^{-{\rm i}\left\langle\cc\ca\cs(n^{(1)})\right\rangle^2(t-s)}|\prod_{j=1}^{P}|\left\{e^{-{\rm i}\langle n_j\rangle^2s}\right\}^{\ast^{[j-1]}}|{\rm d}s
%=&\int_0^t{\rm d}s\\
%=&t\\
=\frac{t^{\ell(1)}}{\cD(1)}.
\end{align*}
%\marginpar{We have changed ``indicates" into ``shows", which we also think is better.}
This proves that \eqref{dd} is true for $k=1$.

Let $k\geq2$. Assume that \eqref{dd} holds for all $1<k^\prime<k$.

For $k, (\gamma_j^{(k-1)})_{1\leq j\leq P}=\gamma^{(k)}\in(\Gamma^{(k-1)})^{P}$ and $ (n_j^{(k-1)})_{1\leq j\leq P}\in\prod_{j=1}^{P}\cN^{(k-1,\gamma_j^{(k-1)})}$, according to the definition of $\cI$ and the induction hypothesis, one can derive that
\begin{align*}
|\cI^{(k,\gamma^{(k)})}(t,n^{(k)})|
%=&|\int_0^te^{-{\rm i}\left\langle\cc\ca\cs(n^{(k)})\right\rangle^2(t-s)}\prod_{j=1}^{P}\left\{\cI^{(k-1,\gamma_j^{(k-1)})}(s,n_j^{(k-1)})\right\}^{\ast^{[j-1]}}{\rm d}s|\\
%\leq&\int_0^t\prod_{j=1}^{P}|\left\{\cI^{(k-1,\gamma_j^{(k-1)})}(s,n_j^{(k-1)})\right\}^{\ast^{[j-1]}}|{\rm d}s\\
=&\int_0^t\prod_{j=1}^{P}|\cI^{(k-1,\gamma_j^{(k-1)})}(s,n_j^{(k-1)})|{\rm d}s\\
\leq&\int_0^t\prod_{j=1}^{P}\frac{s^{\ell(\gamma_j^{(k-1)})}}{\cD(\gamma_j^{(k-1)})}{\rm d}s\\
%=&\frac{t^{1+\sum_{j=1}^{P}\ell(\gamma_j^{(k-1)})}}{\left(1+\sum_{j=1}^{P}\ell(\gamma_j^{(k-1)})\right)\prod_{j=1}^{P}\cD(\gamma_j^{(k-1)})}\\
=&\frac{t^{\ell(\gamma^{(k)})}}{\cD(\gamma^{(k)})}.
\end{align*}
Here we use \ref{ee}. This shows that \eqref{dd} is true for $k$.

By induction, \eqref{dd} is true for all $k\in\bN$. The completes the proof of Lemma \ref{li}.
\end{proof}

Finally, {an estimate} of $\cF$ is the following Lemma \ref{fl}.

%measure the size of  and the estimates result is the following Lemma \ref{fl}.
\begin{lemm}\label{fl}
For all $k\geq1$, we have
\begin{align}
|\cF^{(k,\gamma^{(k)})}(n^{(k)})|\leq1.
\end{align}
\end{lemm}
\begin{proof}
This is easily obtained by induction.
\end{proof}

%\marginpar{We have removed ``Above all". ``yields the size of" have been changed into ``yields an estimate on the size of".}
Putting {Lemmas} \ref{cl}, \ref{li} and \ref{fl} together yields an estimate on the size of $c_k$; see the following Lemma \ref{pl}.
\begin{lemm}\label{pl}
For all $k\geq1$, the Picard sequence $c_k(t,n)$ is $\mathtt r/2$-polynomially decaying. To be more precise, for all $k\geq1$, we have
\begin{align}\label{pdd}
|c_k(t,n)|\leq B(1+|n|)^{-\frac{\mathtt r}{2}},
\end{align}
where $B=\max\{A^{\frac{1}{2p}},A^{\frac{1}{2p}}P(2p)^{-1}
\mathfrak b(\mathtt r/2;\nu)\}$.
\end{lemm}
\begin{proof}
First, it follows from Lemma \ref{lee}, {Lemmas} \ref{cl}-\ref{fl}  and Proposition \ref{propsl} that
\begin{align*}
&|c_k(t,n)|\\
%\leq&\sum_{\gamma^{(k)}\in\Gamma^{(k)}}\sum_{\substack{n^{(k)}\in\cN^{(k,\gamma^{(k)})}\\\cc\ca\cs(n^{(k)})=n}}
%|\cC^{(k,\gamma^{(k)})}(n^{(k)})||\cI^{(k,\gamma^{(k)})}(t,n^{(k)})||\cF^{(k,\gamma^{(k)})}(n^{(k)})|\\
\leq&\sum_{\gamma^{(k)}\in\Gamma^{(k)}}\sum_{\substack{n^{(k)}\in\cN^{(k,\gamma^{(k)})}\\\cc\ca\cs(n^{(k)})=n}}A^{\sigma(\gamma^{(k)})}\prod_{j=1}^{2p\sigma(\gamma^{(k)})} \left(1+|(n^{(k)})_j|\right)^{-\mathtt r}\cdot \frac{t^{\ell(\gamma^{(k)})}}{\cD(\gamma^{(k)})}\\
%=&A^{\frac{1}{2p}}\sum_{\gamma^{(k)}\in\Gamma^{(k)}}\frac{(At)^{\ell(\gamma^{(k)})}}{\cD(\gamma^{(k)})}\sum_{\substack{n^{(k)}\in\cN^{(k,\gamma^{(k)})}\\\cc\ca\cs(n^{(k)})=n}}\prod_{j=1}^{2p\sigma(\gamma^{(k)})} \left(1+|(n^{(k)})_j|\right)^{-\mathtt r}\\
=&A^{\frac{1}{2p}}\sum_{\gamma^{(k)}\in\Gamma^{(k)}}\frac{(At)^{\ell(\gamma^{(k)})}}{\cD(\gamma^{(k)})}\sum_{\substack{n^{(k)}\in\cN^{(k,\gamma^{(k)})}\\\cc\ca\cs(n^{(k)})=n}}\prod_{j=1}^{2p\sigma(\gamma^{(k)})} \left(1+|(n^{(k)})_j|\right)^{-\mathtt r/2}\cdot\prod_{j=1}^{2p\sigma(\gamma^{(k)})} \left(1+|(n^{(k)})_j|\right)^{-\mathtt r/2}.
\end{align*}

On the one hand, under the {splitting} condition $\cc\ca\cs(n^{(k)})=n$, applying the generalized Bernoulli inequality \eqref{mie} and the triangle inequality of $\ell^1$-norm \eqref{ine2} yields that
\begin{align*}
\prod_{j=1}^{2p\sigma(\gamma^{(k)})} \left(1+|(n^{(k)})_j|\right)^{-\mathtt r/2}
%=&\left(\prod_{j=1}^{2p\sigma(\gamma^{(k)})}\left(1+|(n^{(k)})_j|\right)\right)^{-\mathtt r/2}\\
\leq&\left(1+\sum_{j=1}^{2p\sigma(\gamma^{(k)})}|(n^{(k)})_j|\right)^{-\mathtt r/2}\\
=&\left(1+\sum_{j=1}^{2p\sigma(\gamma^{(k)})}|(-1)^{j-1}(n^{(k)})_j|\right)^{-\mathtt r/2}\\
{\leq}&{\left(1+\left|\sum_{j=1}^{2p\sigma(\gamma^{(k)})}(-1)^{j-1}(n^{(k)})_j\right|\right)^{-\mathtt r/2}}\\
=&\left(1+|\cc\ca\cs(n^{(k)})|\right)^{-\mathtt r/2}\\
=&(1+|n|)^{-\mathtt r/2}.
\end{align*}
%The second equal sigh from back to front  is  from \eqref{as2}.

On the other hand, it follows from \eqref{dimsi}, \eqref{mm}, \eqref{hb} and Lemma \ref{lemmriem} that
\begin{align*}
\sum_{\substack{n^{(k)}\in\cN^{(k,\gamma^{(k)})}\\\cc\ca\cs(n^{(k)})=n}}\prod_{j=1}^{2p\sigma(\gamma^{(k)})} \left(1+|(n^{(k)})_j|\right)^{-\mathtt r/2}\leq&\sum_{n^{(k)}\in(\bZ^\nu)^{2p\sigma(\gamma^{(k)})}}\prod_{j=1}^{2p\sigma(\gamma^{(k)})} \left(1+|(n^{(k)})_j|\right)^{-\mathtt r/2}\\
%=&\prod_{j=1}^{2p\sigma(\gamma^{(k)})}\sum_{m_j\in\bZ^\nu}(1+|m_j|)^{-\mathtt r/2}\\
=&\prod_{j=1}^{2p\sigma(\gamma^{(k)})}\mathscr H(\mathtt r/2;\nu)\\
%\leq&\prod_{j=1}^{2p\sigma(\gamma^{(k)})}\mathfrak b (\mathtt r/2;\nu)\\
\leq&\left(\mathfrak b(\mathtt r/2;\nu)\right)^{2p\sigma(\gamma^{(k)})}.
%\\
%=&\mathfrak b(\mathtt r/2;\nu)\cdot\left(\mathfrak b(\mathtt r/2;\nu)\right)^{2p\ell(\gamma^{(k)})}.
\end{align*}

Furthermore, one can derive that
\begin{align*}
|c_k(t,n)|\leq&(1+|n|)^{-\mathtt r/2}\cdot A^{\frac{1}{2p}}
\mathfrak b(\mathtt r/2;\nu)\cdot\sum_{\gamma^{(k)}\in\Gamma^{(k)}}\frac{\left(A\left(\mathfrak b(\mathtt r/2;\nu)\right)^{2p}t\right)^{\ell(\gamma^{(k)})}}{\cD(\gamma^{(k)})}.
\end{align*}

Finally, it follows from Lemma \ref{gal} that if $0<A\left(\mathfrak b(\mathtt r/2;\nu)\right)^{2p}t\leq\frac{(2p)^{2p}}{P^{P}}$, that is,
\begin{align}\label{t0}
0<t\leq\frac{(2p)^{2p}}{A\left(\mathfrak b(\mathtt r/2;\nu)\right)^{2p}P^{P}}\triangleq t_0,
\end{align}
then
\begin{align}\label{pdc2}
|c_k(t,n)|\leq B(1+|n|)^{-\mathtt r/2},\quad\forall n\in\bZ^\nu,
\end{align}
where $B=A^{\frac{1}{2p}}\frac{P}{2p}
\mathfrak b(\mathtt r/2;\nu)$. This shows that the Picard sequence $c_k(t,n)$ is $\mathtt r/2$-polynomially decaying on a suitable time interval $(0,t_0]$. This completes the proof of Lemma \ref{pl}.
\end{proof}

\subsection{The Picard Sequence is a Cauchy Sequence}

In this subsection, we  will use the original form \eqref{ck} of the Picard sequence and its polynomial decay property to prove that it is a Cauchy sequence.

We first estimate {the difference between consecutive terms of the Picard sequence} $\{c_k(t,n)\}$ and obtain the following Lemma \ref{nl}.
\begin{lemm}\label{nl}
For all $t\in(0,t_0], n\in\bZ^\nu$, and $k\geq1$, we have
\begin{align}
\label{alp}|c_{k}(t,n)-c_{k-1}(t,n)|\leq&\frac{P^{k-1}B^{(P-1)k+1}t^k}{k!}\sum_{\substack{n_j\in\bZ^\nu,j=1,\cdots,(P-1)k+1\\\sum_{j=1}^{(P-1)k+1}(-1)^{j-1}n_j=n}}\left\{\prod_{j=1}^{(P-1)k+1}(1+|n_j|)\right\}^{-\mathtt r/2}\\
\label{alq}\leq&\frac{B\mathfrak b(\mathtt r/4;\nu)}{P}\cdot\frac{\left\{P\left(B\mathfrak b(\mathtt r/4;\nu)\right)^{P-1}t\right\}^{k}}{k!}\cdot(1+|n|)^{-\mathtt r/4}.
\end{align}
This implies that $\{c_k(t,n)\}_{k\in\mathbb N}$ is a Cauchy sequence.
\end{lemm}
\begin{proof}
For $k=1$, it follows from the definition of \eqref{ck} and the polynomial decay condition \eqref{pdd} for the initial Fourier data $c$ that
\begin{align*}
|c_1(t,n)-c_0(t,n)| %& = \Big| {\rm i}\int_0^te^{-{\rm i}\langle n\rangle^2(t-s)}\sum_{\substack{n_j\in\bZ^\nu,~~j=1,\cdots,P\\\sum_{j=1}^{P}(-1)^{j-1}n_j=n}}
%\prod_{j=1}^{P}\left\{c_{0}(s,n_j)\right\}^{\ast^{[j-1]}}{\rm d}s \Big| \\
%\leq&\int_0^t\sum_{\substack{n_j\in\bZ^\nu,~~j=1,\cdots,P\\\sum_{j=1}^{P}(-1)^{j-1}n_j=n}}
%\prod_{j=1}^{P}\left|\left\{c_{0}(s,n_j)\right\}^{\ast^{[j-1]}}\right|{\rm d}s\\
& \leq \int_0^t\sum_{\substack{n_j\in\bZ^\nu,~~j=1,\cdots,P\\\sum_{j=1}^{P}(-1)^{j-1}n_j=n}}
\prod_{j=1}^{P}\left|c_{0}(s,n_j)\right|{\rm d}s\\
%& = \int_0^t\sum_{\substack{n_j\in\bZ^\nu,~~j=1,\cdots,P\\\sum_{j=1}^{P}(-1)^{j-1}n_j=n}}
%\prod_{j=1}^{P}\left|c(n_j)\right|{\rm d}s\\
%& \leq t \cdot\sum_{\substack{n_j\in\bZ^\nu,~~j=1,\cdots,P\\\sum_{j=1}^{P}(-1)^{j-1}n_j=n}}
%\prod_{j=1}^{P}B(1+|n_j|)^{-\mathtt r/2}\\
& = B^{P}t\sum_{\substack{n_j\in\bZ^\nu,~~j=1,\cdots,P\\\sum_{j=1}^{P}(-1)^{j-1}n_j=n}}
\left\{\prod_{j=1}^{P}(1+|n_j|)\right\}^{-\mathtt r/2}.
\end{align*}
This shows that \eqref{alp} is true for $k=1$.

Let $k\geq2$. Assume that \eqref{alp} holds for all $1<k^\prime<k$.

By the definition of \eqref{ck} and the polynomial decay property \eqref{pdd} along with the following decomposition
\begin{align}\label{deco}
|\prod_{j=1}^{j_0} a_j-\prod_{j=1}^{j_0} b_j|
\leq\sum_{J=1}^{j_0}\prod_{j=1}^{J-1}|b_j|\cdot|a_{J}-
b_{J}|\cdot\prod_{j=J+1}^{{j_0}}|a_j|,
\end{align}
where
\begin{align}\label{deco2}
\prod_{j=1}^0|b_j|:=1\quad\text{and}\quad\prod_{j={j_0}+1}^{j_0}|a_j|:=1,
\end{align}
 one can derive that
\begin{align*}
&|c_k(t,n)-c_{k-1}(t,n)|\\
%=& \Big| {\rm i}\int_0^te^{{\rm i}\langle n\rangle^2(t-s)}\sum_{\substack{n_j\in\bZ^\nu,j=1,\cdots,P\\\sum_{j=1}^{P}(-1)^{j-1}n_j=n}}
%\left\{\prod_{j=1}^{P}\{c_{k-1}(s,n_j)\}^{\ast^{[j-1]}}-\prod_{j=1}^{P}\{c_{k-2}(s,n_j)\}^{\ast^{[j-1]}}\right\}
%{\rm d}s \Big| \\
\leq&\int_0^t\sum_{\substack{n_j\in\bZ^\nu,j=1,\cdots,P\\\sum_{j=1}^{P}(-1)^{j-1}n_j=n}}|\prod_{j=1}^{P}
\{c_{k-1}(s,n_j)\}^{\ast^{[j-1]}}-\prod_{j=1}^{P}\{c_{k-2}(s,n_j)\}^{\ast^{[j-1]}}|{\rm d}s\\
\leq&\sum_{J=1}^{P}\int_0^t\sum_{\substack{n_j\in\bZ^\nu,j=1,\cdots,P\\\sum_{j=1}^{P}(-1)^{j-1}n_j=n}}\prod_{j=1}^{J-1}
|\{c_{k-2}(s,n_j)\}^{\ast^{[j-1]}}|\cdot\prod_{j=J+1}^{P}|\{c_{k-1}(s,n_j)\}^{\ast^{[j-1]}}|\\
&\cdot|\{c_{k-1}(s,n_J)\}^{\ast^{[J-1]}}-\{c_{k-2}(s,n_J)\}^{\ast^{[J-1]}}|{\rm d}s\\
%=&\sum_{J=1}^{P}\int_0^t\sum_{\substack{n_j\in\bZ^\nu,j=1,\cdots,P\\\sum_{j=1}^{P}(-1)^{j-1}n_j=n}}\prod_{j=1}^{J-1}
%|c_{k-2}(s,n_j)|\cdot\prod_{j=J+1}^{P}|c_{k-1}(s,n_j)|\cdot|c_{k-1}(s,n_J)-c_{k-2}(s,n_J)|{\rm d}s\\
\leq&\sum_{J=1}^{P}\int_0^t\sum_{\substack{n_j\in\bZ^\nu,j=1,\cdots,P\\\sum_{j=1}^{P}(-1)^{j-1}n_j=n}}\prod_{j=1}^{J-1}B(1+|n_j|)^{-\mathtt r/2}\cdot\prod_{j=J+1}^{P}B(1+|n_j|)^{-\mathtt r/2}\\
&\cdot\frac{P^{k-2}B^{(P-1)(k-1)+1}s^{k-1}}{(k-1)!}\sum_{\substack{m_j\in\bZ^\nu,j=1,\cdots,(P-1)(k-1)+1\\\sum_{j=1}^{(P-1)(k-1)+1}(-1)^{j-1}m_j=n_J}}\left\{\prod_{j=1}^{(P-1)(k-1)+1}(1+|m_j|)\right\}^{-\mathtt r/2}{\rm d}s\\
=&\sum_{J=1}^{P}\frac{P^{k-2}B^{(P-1)k+1}t^k}{k!}\sum_{\substack{q_j\in\bZ^\nu,j=1,\cdots,(P-1)k+1\\\sum_{j=1}^{(P-1)k+1}(-1)^{j-1}q_j=n}}\prod_{j=1}^{J-1}(1+|q_j|)^{-\mathtt r/2}\\
&\cdot\prod_{j=J+1+(P-1)(k-1)}^{(P-1)k+1}(1+|q_j|)^{-\mathtt r/2}
\cdot\prod_{j=(J-1)+1}^{(J-1)+(P-1)(k-1)+1}(1+|q_j|)^{-\mathtt r/2}\\
=&\frac{P^{k-1}B^{(P-1)k+1}t^k}{k!}\sum_{\substack{q_j\in\bZ^\nu,j=1,\cdots,(P-1)k+1\\\sum_{j=1}^{(P-1)k+1}(-1)^{j-1}q_j=n}}\left\{\prod_{j=1}^{(P-1)k+1}(1+|q_j|)\right\}^{-\mathtt r/2}.
\end{align*}
Here we use the following decomposition and permutation:
\begin{align}\label{dp}
n=&\sum_{j=1}^{P}(-1)^{j-1}n_j\nonumber\\
=&\sum_{j=1}^{J-1}(-1)^{j-1}n_j+(-1)^{J-1}n_J+\sum_{j=J+1}^{P}(-1)^{j-1}n_j\nonumber\\
=&\sum_{j=1}^{J-1}(-1)^{j-1}n_j+(-1)^{J-1}\sum_{j=1}^{(P-1)(k-1)+1}(-1)^{j-1}m_j+\sum_{j=J+1}^{P}(-1)^{j-1}n_j\nonumber\\
=&\sum_{j=1}^{J-1}(-1)^{j-1}n_j+(-1)^{J-1}\sum_{j=1}^{(P-1)(k-1)+1}(-1)^{j-1}m_j+\sum_{j=J+1}^{P}(-1)^{j-1}n_j\nonumber\\
=&\sum_{j=1}^{J-1}(-1)^{j-1}q_j+\sum_{j=(J-1)+1}^{(J-1)+(P-1)(k-1)+1}(-1)^{j-1}q_j+\sum_{j=J+1+(P-1)(k-1)}^{(P-1)k+1}(-1)^{j-1}q_j\nonumber\\
=&\sum_{j=1}^{(P-1)k+1}(-1)^{j-1}q_j,
\end{align}
where
\begin{align}\label{dp2}
q_j=
\begin{cases}
n_j,&1\leq j\leq J-1;\\
m_{-(J-1)+j},&(J-1)+1\leq j\leq(J-1)+(P-1)(k-1)+1;\\
n_{-(P-1)(k-1)+j},&(J-1)+(P-1)(k-1)+1+1\leq j\leq (P-1)k+1.
\end{cases}
\end{align}
This shows that \eqref{alp} holds for $k$.

By induction, \eqref{alp} is true for all $k\in\bN$. This completes the proof of \eqref{alp}.

Furthermore, it follows from \eqref{alp} and the generalized Bernoulli inequality \eqref{gbi} that
\begin{align*}
&\sum_{\substack{q_j\in\bZ^\nu,j=1,\cdots,(P-1)k+1\\\sum_{j=1}^{(P-1)k+1}(-1)^{j-1}q_j=n}}\left\{\prod_{j=1}^{(P-1)k+1}(1+|q_j|)\right\}^{-\mathtt r/2}\\
%=&\sum_{\substack{q_j\in\bZ^\nu,j=1,\cdots,(P-1)k+1\\\sum_{j=1}^{(P-1)k+1}(-1)^{j-1}q_j=n}}\prod_{j=1}^{(P-1)k+1}(1+|q_j|)^{-\mathtt r/4}\cdot\left\{\prod_{j=1}^{(P-1)k+1}(1+|q_j|)\right\}^{-\mathtt r/4}\\
\leq&\sum_{\substack{q_j\in\bZ^\nu,j=1,\cdots,(P-1)k+1\\\sum_{j=1}^{(P-1)k+1}(-1)^{j-1}q_j=n}}\prod_{j=1}^{(P-1)k+1}(1+|q_j|)^{-\mathtt r/4}\cdot\left\{1+\sum_{j=1}^{(P-1)k+1}|(-1)^{j-1}q_j|\right\}^{-\mathtt r/4}\\
%\leq&(1+|n|)^{-\mathtt r/4}\sum_{\substack{q_j\in\bZ^\nu,j=1,\cdots,(P-1)k+1\\\sum_{j=1}^{(P-1)k+1}(-1)^{j-1}q_j=n}}\prod_{j=1}^{(P-1)k+1}(1+|q_j|)^{-\mathtt r/4}\\
\leq&(1+|n|)^{-\mathtt r/4}\sum_{\substack{q_j\in\bZ^\nu\\j=1,\cdots,(P-1)k+1}}\prod_{j=1}^{(P-1)k+1}(1+|q_j|)^{-\mathtt r/4}\\
%=&(1+|n|)^{-\mathtt r/4}\prod_{j=1}^{(P-1)k+1}\sum_{q_j\in\bZ^\nu}(1+|q_j|)^{-\mathtt r/4}\\
%=&(1+|n|)^{-\mathtt r/4}\prod_{j=1}^{(P-1)k+1}\mathscr H(\mathtt r/4;\nu)\\
%\leq&(1+|n|)^{-\mathtt r/4}\prod_{j=1}^{(P-1)k+1}\mathfrak b(\mathtt r/4;\nu)\\
=&\left\{\mathfrak b(\mathtt r/4;\nu)\right\}^{(P-1)k+1}(1+|n|)^{-\mathtt r/4}.
\end{align*}
Hence we have
\begin{align*}
|c_k(t,n)-c_{k-1}(t,n)|\leq\frac{B\mathfrak b(\mathtt r/4;\nu)}{P}\cdot\frac{\left\{P\left(B\mathfrak b(\mathtt r/4;\nu)\right)^{P-1}t\right\}^{k}}{k!}\cdot(1+|n|)^{-\mathtt r/4},\quad \forall k\geq1.
\end{align*}
This completes the proof of \eqref{alq} and hence the proof of Lemma \ref{nl}.
\end{proof}

Next we estimate {the distance between} any two terms of the Picard sequence and obtain the following Lemma \ref{sdf}.
\begin{lemm}\label{sdf}
For any $n\in\bZ^\nu, t\in(0,t_0]$, and all $k,k^\prime\in\bN$, we have
\begin{align*}
|c_{k+k^\prime}(t,n)-c_{k}(t,n)|\leq\frac{B\mathfrak b(\mathtt r/4;\nu)}{P}\sum_{j=1}^{\infty}\frac{\left\{P\left(B\mathfrak b(\mathtt r/4;\nu)\right)^{P-1}t\right\}^{k+j}}{(k+j)!}\cdot(1+|n|)^{-\mathtt r/4}.
\end{align*}
Hence $\{c_k(t,n)\}$ is a Cauchy sequence on $(t,n)\in(0,t_0]\times\bZ^\nu$.
\begin{proof}
It follows from the triangle inequality and Lemma \ref{nl} that
\begin{align*}
|c_{k+k^\prime}(t,n)-c_k(t,n)|&\leq\sum_{j=1}^{k^\prime}|c_{k+j}(t,n)-c_{k+j-1}(t,n)|\\
%&\leq\sum_{j=1}^{k^\prime}\frac{B\mathfrak b(\mathtt r/4;\nu)}{P}\cdot\frac{\left\{P\left(B\mathfrak b(\mathtt r/4;\nu)\right)^{P-1}t\right\}^{k+j}}{(k+j)!}\cdot(1+|n|)^{-\mathtt r/4}\\
%&=\frac{B\mathfrak b(\mathtt r/4;\nu)}{P}\sum_{j=1}^{k^\prime}\frac{\left\{P\left(B\mathfrak b(\mathtt r/4;\nu)\right)^{P-1}t\right\}^{k+j}}{(k+j)!}\cdot(1+|n|)^{-\mathtt r/4}\\
&\leq\frac{B\mathfrak b(\mathtt r/4;\nu)}{P}\sum_{j=1}^{\infty}\frac{\left\{P\left(B\mathfrak b(\mathtt r/4;\nu)\right)^{P-1}t\right\}^{k+j}}{(k+j)!}\cdot(1+|n|)^{-\mathtt r/4}.
%&=\frac{B\mathfrak b(\mathtt r/4;\nu)}{P}\sum_{j=k+1}^{\infty}\frac{\left\{P\left(B\mathfrak b(\mathtt r/4;\nu)\right)^{P-1}t\right\}^{j}}{j!}\cdot(1+|n|)^{-\mathtt r/4}.
\end{align*}
Notice that
\[\sum_{j=1}^{\infty}\frac{\left\{P\left(B\mathfrak b(\mathtt r/4;\nu)\right)^{P-1}t\right\}^{k+j}}{(k+j)!}\]
can be viewed as the $k+1$-th remainder of $\exp\left(P\left(B\mathfrak b(\mathtt r/4;\nu)\right)^{P-1}t\right)$. Hence
$|c_{k+k^\prime}(t,n)-c_{k}(t,n)|$ tends to $0$ uniformly with respect to $k^\prime,t\in(0,t_0]$ and $n\in\bZ^\nu$ by letting $k\rightarrow\infty$. This shows that $\{c_k(t,n)\}_{k\in\mathbb N}$ is a Cauchy sequence on $(t,n)\in(0,t_0]\times\bZ^\nu$. This completes the proof of Lemma \ref{sdf}.
\end{proof}
\end{lemm}
\begin{rema}
%\marginpar{I think that these are the main differences generalized by the complex conjugate, compared with the real cases. So a remark is here.}
It should be emphasized that the above estimates are different from those in \cite{DLX22AR} because we have to deal with the complicated alternating effect generated by the power law nonlinearity $|u|^{2p}u$ in a combinatorial manner.
{To} this end, we introduce {\bf the power of $\ast^{[\cdot]}$}(see \autoref{lc}), {\bf the combinatorial alternating sums $\cc\ca\cs(\cdot)$}(see Definition \ref{defias}), and with these come a series of more complex operations, such as %\autoref{secco1}, \autoref{secco2}, \autoref{secco3},
\eqref{oc}, \eqref{ck}, \eqref{ct}, \eqref{cc}, \eqref{dp}, \eqref{dp2}, \autoref{di}, {Lemmas} \ref{lee}-\ref{nl} and so on and so forth. Without these, {it seems that there is no bridge to arrive at the end of proof}; compare \cite{DLX22AR}.
%\auto\ref{secco1},\autoref{secco2},\autoref{secco3},\ref{p2},\eqref{as1}
\end{rema}

\subsection{Existence and Convergence}

Since the Picard sequence $\{c_k(t,n)\}$ is a Cauchy sequence, there is a limit function, denoted by $\tc(t,n)$, such that $\tc$ is a solution to \eqref{ck} and it satisfies the initial Fourier data $\tc(0,n)=c(n)$ for all $n\in\bZ^\nu$.

Define
\begin{align*}
\mathfrak u(t,x)&=\sum_{n\in\bZ^\nu}\tc(t,n)e^{{\rm i}\langle n\rangle x},\\
(\partial_x^\#\mathfrak u)(t,x)&=\sum_{n\in\bZ^\nu}({\rm i}\langle n\rangle)^{\#}\tc(t,n)e^{{\rm i}\langle n\rangle x},\quad\#=1,2,
\end{align*}
and
\begin{align*}
(\partial_t\mathfrak u)(t,x)&={\rm i}\sum_{n\in\bZ^\nu}\left\{({\rm i}\langle n\rangle)^2
\tc(t,n)+\sum_{\substack{n_j\in\bZ^\nu, j=1,\cdots, P\\\sum_{j=1}^{P}(-1)^{j-1}n_j=n}}\prod_{j=1}^{P}\{\tc(t,n_j)\}^{\ast^{[j-1]}}\right\}e^{{\rm i}\langle n\rangle x}.
\end{align*}
These {spatially quasi-periodic functions} are well-defined. In fact, it follows from the polynomial decay \eqref{pdd} and Lemma \ref{lemmriem} that
\begin{align*}
\sum_{n\in\bZ^\nu}|n|^2|\tc(t,n)|&\lesssim_{B}\sum_{n\in\bZ^\nu}|n|^2(1+|n|)^{-\mathtt r/2}\\
&\leq\sum_{n\in\bZ^\nu}(1+|n|)^{2-\frac{\mathtt r}{2}}\\
%&=\mathscr H\left(\frac{\mathtt r}{2}-2;\nu\right)\\
&\leq \mathfrak b\left(\frac{\mathtt r}{2}-2;\nu\right)
\end{align*}
and
\begin{align*}
\sum_{n\in\bZ^\nu}\sum_{\substack{n_j\in\bZ^\nu, j=1,\cdots, P\\\sum_{j=1}^{P}(-1)^{j-1}n_j=n}}\prod_{j=1}^{P}|
\tc(t,n_j)|&\lesssim_{B,P}\sum_{n\in\bZ^\nu}\sum_{\substack{n_j\in\bZ^\nu, j=1,\cdots, P\\\sum_{j=1}^{P}(-1)^{j-1}n_j=n}}\prod_{j=1}^{P}(1+|n_j|)^{-\mathtt r/2}\\
&\leq\sum_{n\in\bZ^\nu}(1+|n|)^{-\mathtt r/4}\sum_{n_j\in\bZ^\nu,j=1,\cdots,P}\prod_{j=1}^{P}(1+|n_j|)^{-\mathtt r/4}\\
%&=\sum_{n\in\bZ^\nu}(1+|n|)^{-\mathtt r/4}\prod_{j=1}^P\sum_{n_j\in\bZ^\nu}(1+|n_j|)^{-\mathtt r/4}\\
%&=\left\{\mathscr H(\mathtt r/4;\nu)\right\}^{P+1}\\
&\leq\left\{\mathfrak b(\mathtt r/4;\nu)\right\}^{P+1},
\end{align*}
provided that
\begin{align}\label{rnu}
2\leq\nu<\min\left\{\frac{\mathtt r}{2}-2,\frac{\mathtt r}{4}\right\}=\frac{\mathtt r}{4}.
\end{align}
As a result, $\mathfrak u$ is a classical solution to the quasi-periodic Cauchy problem \eqref{cnls}-\eqref{id0}.
%This completes the proof of Theorem ???

\subsection{Uniqueness}

In this subsection, we prove the uniqueness result, with arbitrary $p\in\bN$, in the case of a polynomial decay condition.

Let
\begin{align}
u^1(t,x)&=\sum_{n\in\bZ^\nu}c^1(t,n)e^{{\rm i}\langle n\rangle x} \quad\text{and}\quad u^2(t,x)=\sum_{n\in\bZ^\nu}c^2(t,n)e^{{\rm i}\langle n\rangle x}
\end{align}
be two solutions to \eqref{cnls} defined on $(t,x)\in[0,t_0]\times\bR$. Assume that
\begin{itemize}
\item (with the same initial data)~~$c^1(0,n)=c^2(0,n)=c(n)$ for all $n\in\bZ^\nu$;
  \item (polynomial decay condition)\begin{align*}
\max\{|c^1(t,n)|,|c^2(t,n)|\}\leq B(1+|n|)^{-\frac{\mathtt r}{2}},\quad n\in\bZ^\nu~\text{and}~0\leq t\leq t_0.
 \end{align*}
\end{itemize}

%
%  with the same initial data
% \[u^{1}(0,x)=\sum_{n\in\bZ^\nu}c(n)e^{{\rm i}\langle n\rangle x}=u^2(0,x).\]

Our goal is to estimate the difference $|c^1(t,{n})-c^2(t,{n})|$.

In the Fourier space, $c^1$ and $c^2$ respectively satisfy the following integral equations,
 \begin{align*}
c^1(t,n)=e^{-{\rm i}\langle n\rangle^2t}c(n)+{\rm i}\int_0^te^{-{\rm i}\langle n\rangle^2(t-s)}\sum_{\substack{n_j\in\bZ^\nu,~~j=1,\cdots,P\\\sum_{j=1}^{P}(-1)^{j-1}n_j=n}}
\prod_{j=1}^{P}\left\{c^1(s,n_j)\right\}^{\ast^{[j-1]}}{\rm d}s,\\
c^2(t,n)=e^{-{\rm i}\langle n\rangle^2t}c(n)+{\rm i}\int_0^te^{-{\rm i}\langle n\rangle^2(t-s)}\sum_{\substack{n_j\in\bZ^\nu,~~j=1,\cdots,P\\\sum_{j=1}^{P}(-1)^{j-1}n_j=n}}
\prod_{j=1}^{P}\left\{c^2(s,n_j)\right\}^{\ast^{[j-1]}}{\rm d}s.
 \end{align*}

%The main result is the following Lemma \ref{sp}.

%\begin{lemm}\label{sp}
For all $1\leq k\in\bN$,  by induction and the polynomial decay condition, one can obtain that
\begin{align}
|c^1(t,n)-c^2(t,n)|&\label{uu1}\leq2B^{(P-1)k+1}P^k\cdot\frac{t^k}{k!}\sum_{\substack{n_\in\bZ^\nu,j=1,\cdots,(P-1)k+1\\\sum_{j=1}^{(P-1)k+1}(-1)^{j-1}n_j=n}}\prod_{j=1}^{(P-1)k+1}(1+|n_j|)^{-\frac{\mathtt r}{2}}\\
&\label{uu2}\leq2B\mathfrak b(\mathtt r/2;\nu)\cdot\frac{\left\{(B\mathfrak b(\mathtt r/2;\nu))^{P-1}Pt\right\}^k}{k!}.
\end{align}

The proof here is similar to the one in the proof of Cauchy sequence.

We first prove \eqref{uu1}. For $k=1$, it follows from decomposition \eqref{deco}-\eqref{deco2} and the polynomial decay property \eqref{pdc2} that
\begin{align*}
|c^1(t,n)-c^2(t,n)|%\leq&\int_0^t\sum_{\substack{n_j\in\bZ^\nu,j=1,\cdots,P\\\sum_{j=1}^{P}(-1)^{j-1}n_j=n}}|\prod_{j=1}^{P}\{c^1(s,n_j)\}^{\ast^{[j-1]}}-\prod_{j=1}^{P}\{c^2(s,n_j)\}^{\ast^{[j-1]}}|{\rm d}s\\
\leq&\sum_{J=1}^{P}\int_0^t\sum_{\substack{n_j\in\bZ^\nu,j=1,\cdots,P\\\sum_{j=1}^{P}(-1)^{j-1}n_j=n}}\prod_{j=1}^{J-1}|\{c^1(s,n_j)\}^{\ast^{j-1}}|\cdot\prod_{j=J+1}^{P}|\{c^2(s,n_j)\}^{\ast^{[j-1]}}|\\
&|\{c^1(s,n_j)\}^{\ast^{[j-1]}}-\{c^2(s,n_j)\}^{\ast^{[j-1]}}|{\rm d}s\\
\leq&\sum_{J=1}^{P}\int_0^t\sum_{\substack{n_j\in\bZ^\nu,j=1,\cdots,P\\\sum_{j=1}^{P}(-1)^{j-1}n_j=n}}\prod_{j=1,j\neq J}^{P}B(1+|n_j|)^{-\mathtt r/2}\cdot 2B(1+|n_J|)^{-\mathtt r/2}{\rm d}s\\
=&2B^{P}Pt\sum_{\substack{n_j\in\bZ^\nu,j=1,\cdots,P\\\sum_{j=1}^{P}(-1)^{j}n_j=n}}\prod_{j=1}^{P}(1+|n_j|)^{-\mathtt r/2}.
\end{align*}
This  shows that \eqref{uu1} is true for $k=1$.

Let $k\geq2$. Assume that \eqref{uu1} holds for all $1<k^{\prime}<k$.

For $k$, by the same analysis, one can derive that
\begin{align*}
&|c^1(t,n)-c^2(t,n)|\\
\leq&\sum_{J=1}^{P}\int_0^t\sum_{\substack{n_j\in\bZ^\nu,j=1,\cdots,P\\\sum_{j=1}^{P}(-1)^{j-1}n_j=n}}\prod_{j=1}^{J-1}|\{c^1(s,n_j)\}^{\ast^{j-1}}|\cdot\prod_{j=J+1}^{P}|\{c^2(s,n_j)\}^{\ast^{[j-1]}}|\\
&|\{c^1(s,n_j)\}^{\ast^{[j-1]}}-\{c^2(s,n_j)\}^{\ast^{[j-1]}}|{\rm d}s\\
\leq&\sum_{J=1}^{P}\int_0^t\sum_{\substack{n_j\in\bZ^\nu,j=1,\cdots,P\\\sum_{j=1}^{P}(-1)^{j-1}n_j=n}}\prod_{j=1,j\neq J}^{P}B(1+|n_j|)^{-\mathtt r/2}\cdot 2B^{(P-1)(k-1)+1}P^{k-1}\cdot\frac{s^{k-1}}{(k-1)!}\\
&\sum_{\substack{m_j\in\bZ^\nu,j=1,\cdots,(P-1)(k-1)+1\\\sum_{j=1}^{(P-1)(k-1)+1}(-1)^{j-1}m_j=n_J}}\prod_{j=1}^{(P-1)(k-1)+1}(1+|n_j|)^{-\frac{\mathtt r}{2}}{\rm d}s\\
\leq&2B^{(P-1)k+1}P^k\cdot\frac{t^k}{k!}\sum_{\substack{n_j\in\bZ^\nu,j=1,\cdots,(P-1)k+1\\\sum_{j=1}^{(P-1)k+1}(-1)^{j-1}n_j=n}}\prod_{j=1}^{(P-1)k+1}(1+|n_j|)^{-\mathtt r/2}.
\end{align*}
Here we use again the permutation \eqref{dp}-\eqref{dp2}. This {shows} that \eqref{uu1} is true for $k$.

By induction, we know that \eqref{uu1} holds for all $k\in\bN$.

Furthermore, by Lemma \ref{lemmriem}, one can derive that
\begin{align*}
|c^1(t,n)-c^2(t,n)|%\leq2B^{(P-1)k+1}P^k\cdot\frac{t^k}{k!}\cdot\left\{\mathfrak b(\mathtt r/2;\nu)\right\}^{(P-1)k+1}\\
\leq2B\mathfrak b(\mathtt r/2;\nu)\cdot\frac{\left\{\left(B\mathfrak b(\mathtt r/2;\nu)\right)^{P-1}Pt\right\}^k}{k!}.
\end{align*}
This completes the proof of \eqref{uu2}.

As a result, $|c^1(t,n)-c^2(t,n)|$ tends to zero uniformly in $(t,n)\in[0,t_0]\times\bZ^\nu$ by letting $k\rightarrow\infty$ provided that $2\leq\nu<\mathtt r/2$, this is guaranteed by \eqref{rnu}. This shows that $u^1(t,x)\equiv u^2(t,x)$ for all $(t,x)\in[0,t_0]\times\bR$.

\subsection{Asymptotic Dynamics}
%\marginpar{{\color{magenta}We added the proof of Sobolev asymptoticity here. Please take a look.} {\color{blue}Thank you, looks good.}}
In this subsection, we prove that, for the weakly NLS equation \eqref{ecnls}, within the given time scale, the nonlinear solution will be asymptotic to the associated linear solution in the sense of both the sup-norm $\|\cdot\|_{L_x^\infty(\mathbb R)}$ and the Sobolev-norm $\|\cdot\|_{H^s_x(\mathbb R)}$.

Clearly the linear solution is given by the following Fourier series
\begin{align}\label{linears}
u_{\text{linear}}(t,x)=\sum_{n\in\mathbb Z^\nu}e^{-{\rm i}\langle n\rangle^2t}c(n)e^{{\rm i}\langle n\rangle x}.
\end{align}

For the asymptotic dynamics in the sense of the sup-norm $\|\cdot\|_{L_x^\infty(\mathbb R)}$, it follows from \eqref{ses}, \eqref{ie}, \eqref{linears}, the definition of the $L_x^\infty(\mathbb R)$-norm, and the uniform-in-time decay of the Fourier coefficients that
\begin{align*}
\|(u-u_{\text{linear}})(t)\|_{L_x^\infty(\mathbb R)}&\leq|\epsilon|\sum_{n\in\mathbb Z^\nu}\int_0^t\sum_{\substack{n_j\in\bZ^\nu,~~j=1,\cdots,P\\\sum_{j=1}^{P}(-1)^{j-1}n_j=n}}
\prod_{j=1}^{P}|c(\tau,n_j)|{\rm d}\tau\\
&\leq|\epsilon|t\sum_{n_j\in\mathbb Z^\nu, j=1,\cdots, P}\prod_{j=1}^{P}(1+|n_j|)^{-\frac{\mathtt r}{2}}\\
&\lesssim |\epsilon|^{\eta},
\end{align*}
where $t=\epsilon^{-1+\eta}$ with $0<\eta\ll1$. This implies that
\[\|(u-u_{\text{linear}})(t)\|_{L_x^\infty(\mathbb R)}\rightarrow0,\quad{\text{as}}~\epsilon\rightarrow0.\]

For the asymptotic dynamics in the sense of the Sobolev-norm $\|\cdot\|_{H^s_x(\mathbb R)}$, it follows from \eqref{ses}, \eqref{ie}, \eqref{linears}, the definition of the $H_x^s(\mathbb R)$-norm, and the uniform-in-time decay of the Fourier coefficients that
\begin{align*}
\|(u-u_{\text{linear}})(t)\|_{H_x^s(\mathbb R)}^2%&=\sum_{m=0}^{s}\|\partial_x^m(u-u_{\text{linear}})(t)\|_{L^2(\mathbb R)}^2\\
&=\sum_{m=0}^{s}\|\langle n\rangle^m{(\widehat u-\widehat{u_{\text{linear}}})(t)}\|^2_{\ell_n^2(\mathbb Z^\nu)}\\
&=\sum_{m=0}^s\sum_{n\in\mathbb Z^\nu}\langle n\rangle^{2m}|{\rm i}\epsilon\int_{0}^{t}e^{-{\rm i}\langle n\rangle^2(t-\tau)}\sum_{\substack{n_j\in\bZ^\nu,~~j=1,\cdots,P\\\sum_{j=1}^{P}(-1)^{j-1}n_j=n}}
\prod_{j=1}^{P}\{c(\tau,n_j)\}^{\ast^{[j-1]}}{\rm d}\tau|^2\\
&\leq\epsilon^2\sum_{m=0}^s|\omega|^{2m}\sum_{n\in\mathbb Z^\nu}|n|^{2m}\left\{\int_0^t\sum_{\substack{n_j\in\bZ^\nu,~~j=1,\cdots,P\\\sum_{j=1}^{P}(-1)^{j-1}n_j=n}}
\prod_{j=1}^{P}|c(\tau,n_j)|{\rm d}\tau\right\}^2\\
&\lesssim(\epsilon t)^2\sum_{m=0}^s|\omega|^{2m}\sum_{n\in\mathbb Z^\nu}|n|^{2m}\left\{\sum_{\substack{n_j\in\bZ^\nu,~~j=1,\cdots,P\\\sum_{j=1}^{P}(-1)^{j-1}n_j=n}}
\prod_{j=1}^{P}(1+|n_j|)^{-\mathtt r/2}\right\}^2\\
&\leq(\epsilon t)^2\sum_{m=0}^s|\omega|^{2m}\left\{
\prod_{j=1}^{P}
\underbrace{\sum_{\substack{n_j\in\bZ^\nu}}(1+|n_j|)^{-\mathtt r/4}}_{\mathscr H(\mathtt r/4;\nu)}
\right\}^2
\underbrace{\sum_{n\in\mathbb Z^\nu}(1+|n|)^{2m-\frac{\mathtt r}{2}}}_{\mathscr H(\mathtt r/2-2m;\nu)}\\
&\lesssim (\epsilon t)^2\sum_{m=0}^{s}|\omega|^{2m}\mathfrak b(\mathtt r/2-2m;\nu)\\
&\lesssim|\epsilon|^{2\eta},
\end{align*}
provided that $\mathtt r/2-2m>\nu$ for all $m=0,1,\cdots,s$, that is, $s<\mathtt r/4-\nu/2$. This shows that
\[\|(u-u_{\text{linear}})(t)\|_{H_x^s(\mathbb R)}\rightarrow0,\quad\text{as}~\epsilon\rightarrow 0.\]
This completes the proof of Theorem \ref{cnlsth}.

\subsection{Corollary and Remark}

This subsection gives a corollary and a remark.

\begin{coro}
If the initial Fourier data is exponentially decaying, that is, there exist $A>0$ and $0<\kappa\leq1$ such that $|c(n)|\leq A^{1/2p}e^{-\kappa|n|}$ for all $n\in\bZ^\nu$, then the quasi-periodic Cauchy problem has a unique solution which is local in time and retains the same spatial quasi-periodicity. Furthermore, the solution has a uniform (in time) $\kappa/2$-decay rate, and hence it is in the classical sense. In addition, it is analytic in space variable by Lemma \ref{expp}.

%In addition, classical and local (in time)  the solution has uniformly in time exponentially decaying Fourier coefficients (with a slightly worse decay rate), and hence
\end{coro}

%{\color{magenta}
%\begin{rema}
%If $\cp=1$, then the cubic NLS has a unique spatially quasi-periodic solution in the classical sense.
%\end{rema}
%}
\begin{rema}
From the proof for the quasi-periodic Cauchy problem \eqref{cnls}-\eqref{id0}, one can see that the following nonlinear Schr\"odinger
equation\footnote{This type of nonlinear Schr\"odinger equation doesn't enjoy the so-called {\bf gauge invariance} compared with \eqref{cnls}.}
\[{\rm i}\partial_tu+\partial_{xx}u+\lambda|u|^{2p}=0,\quad 1\leq p\in\bN~\text{and}~\lambda=\pm1,\]
with quasi-periodic initial data \eqref{id0} which satisfies the polynomial decay condition \eqref{pd}, {repalcing $A^{\frac{1}{2p}}$ by $A^{\frac{1}{2p-1}}$}, has a unique locally in time spatially quasi-periodic solution with the same frequency vector as the initial data in the classical sense. What's more, this result is true for the exponential decay case and the obtained solution is analytic in the space variable; see Lemma \ref{expp}.
 %(the last statement is similar to the above corollary in this case).

\end{rema}

\section{Appendix}
\begin{lemm}(Geometric-arithmetic mean inequality)\label{lemmga}
Let $\mathtt a_1,\cdots,\mathtt a_n$ be positive numbers, where $n\in\bN$. Then we have the following mean value inequality
\begin{align}\label{mie}
  \left(\prod_{j=1}^n\mathtt a_j\right)^{1/n}\leq\frac{1}{n}\sum_{j=1}^n\mathtt a_j.
  \end{align}
  The left and the right are called the geometric mean and the arithmetic mean of $\mathtt a_1,\cdots,\mathtt a_n$ respectively.
\end{lemm}
%\begin{proof}
%It's obvious that \eqref{mie} is true for $n=1$.
%
%  Let $n\geq2$. Assume that \eqref{mie} holds for all $1<n^\prime<n$.
%
%  For $n$, without loos of generality, we assume that $\mathtt a_n\geq\mathtt a_j$ for all $j=1,\cdots,n-1$. Set $y=\frac{1}{n-1}\sum_{j=1}^{n-1}\mathtt a_j$, then we have $y\leq\mathtt a_n$ (by the definition of $\mathtt a_n$), $y^{n-1}\geq\prod_{j=1}^{n-1}\mathtt a_j$ (induction hypothesis). Furthermore, it follows from the binomial expansion that
%  \begin{align*}
%  \left(\frac{1}{n}\sum_{j=1}^n\mathtt a_j\right)^n&=\left(y+\frac{\mathtt a_n-y}{n}\right)^n\\
%  &\geq n\cdot y^{n-1}\cdot\frac{\mathtt a_n-y}{n}\\
%  &\geq y^{n-1}\mathtt a_n\\
%  &\geq \prod_{j=1}^{n-1}\mathtt a_j\cdot \mathtt a_n\\
%  &=\prod_{j=1}^{n}\mathtt a_j.
%  \end{align*}
%  Hence \eqref{mie} holds for $n$.
%
%  By induction, \eqref{mie} is true for all $n\in\bN$.
%\end{proof}

%\subsubsection{Bound on the Riemann theta function}
\begin{lemm}(Bound on the Riemann zeta function)\label{lemmriem}
Consider the Riemann zeta function
  \begin{align}\label{zeta}
  \zeta(\mathtt s)=\sum_{n=1}^{\infty}\frac{1}{n^{\mathtt s}},\quad\mathtt s\in\bC.
  \end{align}
  If $\mathtt s\in\bR$ and $\mathtt s>1$, then we have the following upper bound estimate
  \begin{align}\label{br}
  \zeta(\mathtt s)\leq1+\frac{1}{\mathtt s-1}.
  \end{align}

  Furthermore,
  for $1<\mathtt s\in\bR$ and $1\leq\nu\in\bN$, where $\mathtt s>\nu$, set
  \begin{align}\label{mm}
  \mathscr H(\mathtt s;\nu)=\sum_{n\in\mathbb Z^\nu}\frac{1}{(1+|n|)^{\mathtt s}}
\end{align}
and
\begin{align}
\mathfrak b(\mathtt s;\nu)=1+\sum_{j_0=1}^{\nu}\left(\begin{matrix}\nu\\j_0\end{matrix}\right)2^{j_0}j_0^{-\mathtt s}\left\{\zeta\left(\frac{\mathtt s}{j_0}\right)\right\}^{j_0}.
  \end{align}
Then we have
\begin{align}\label{hb}
  \mathscr H(\mathtt s;\nu)\leq\mathfrak b(s;\nu),\quad1<\mathtt s\in\bR~\text{and}~1\leq\nu\in\bN.
  \end{align}
\end{lemm}
\begin{lemm}(Generalized Bernoulli inequality)\label{lemmgb}
Let $x_1,\cdots,x_m$ be real numbers, all greater than $-1$, and all with the same sign\footnote{If $x_1,\cdots,x_m$ don't have the same sign, then the generalized Bernoulli inequality is not true. For example, let $n=2, x_1=1/2$ and $x_2=-1/2$, then $(1+x_1)(1+x_2)=3/4$ and $1+x_1+x_2=1$. In this case, the generalized Bernoulli inequality is exactly opposite. Hence the condition of the same sign for $x_1,\cdots,x_n$ is necessary.}. Then we have the following generalized Bernoulli inequality
  \begin{align}\label{gbi}
  \prod_{j=1}^{m}(1+x_j)\geq1+\sum_{j=1}^{m}x_j.
  \end{align}
\end{lemm}
%\begin{proof}
%It's obvious that \eqref{gbi} holds for $m=1$.
%
%      Let $m\geq2$. Assume that it is true for all $1<m^\prime<m$.
%
%      For $m$, one can derive that
%      \begin{align*}
%      \prod_{j=1}^{m}(1+x_j)&=\prod_{j=1}^{m-1}(1+x_j)\cdot(1+x_m)\\
%      &\geq(1+x_1+\cdots+x_{m-1})\cdot (1+x_m)\\
%      &\geq 1+\sum_{j=1}^{m}x_j+\sum_{j=1}^{m-1}x_mx_{j}\\
%      &\geq1+\sum_{j=1}^{m}x_j.
%      \end{align*}
%      The last inequality follows from the same sign condition of $x_1,\cdots,x_m$. This shows that \eqref{gbi} holds for $m$.
%
%      By induction, \eqref{gbi} holds for all $m\in\bN$.
%\end{proof}

%\subsubsection{$\sum_{\gamma^{(k)}\in\Gamma^{(k)}}$}

\begin{lemm}\label{gal}
If $0<\lozenge\leq\frac{(2p)^{2p}}{P^{P}}$, then for all $k\geq1$, we have
\begin{align}\label{la}
\sum_{\gamma^{(k)}\in\Gamma^{(k)}}\frac{\lozenge^{\ell(\gamma^{(k)})}}{\cD(\gamma^{(k)})}\leq\frac{P}{2p}.
\end{align}
\end{lemm}

\begin{proof}
By the definition of the branch set $\Gamma^{(k)}$, we first have the following decomposition,
\begin{align*}
\sum_{\gamma^{(k)}\in\Gamma^{(k)}}\frac{\lozenge^{\ell(\gamma^{(k)})}}{\cD(\gamma^{(k)})}=&\sum_{\gamma^{(k)}=0\in\Gamma^{(k)}}\frac{\lozenge^{\ell(\gamma^{(k)})}}{\cD(\gamma^{(k)})}+\sum_{\gamma^{(k)}\in\Gamma^{(k)}\backslash\{0\}}\frac{\lozenge^{\ell(\gamma^{(k)})}}{\cD(\gamma^{(k)})}\\
\triangleq&(I)_k+(II)_k.
\end{align*}

For all $k\geq1$, $0=\gamma^{(k)}\in\Gamma^{(k)}, \ell(0)=0,\cD(0)=1$, it is obvious that
\begin{align*}
(I)_k=\frac{\lozenge^{\ell(0)}}{\cD(0)}=\frac{\lozenge^{0}}{1}=1.
\end{align*}

For $k=1,1=\gamma^{(1)}\in\Gamma^{(1)},\sigma(1)=1,\cD(1)=1$,  we have
\begin{align*}
(II)_1=\frac{\lozenge^{\ell(1)}}{\cD(1)}=\frac{\lozenge^{1}}{1}=\lozenge.
\end{align*}
Hence
\begin{align*}
\text{the left-hand side of}~~\eqref{la}~~\text{for}~~k=1~~\text{is}~~&1+\lozenge
\leq1+\frac{(2p)^{2p}}{P^{P}}
\leq\frac{P}{2p}.
\end{align*}
This shows that \eqref{la} holds for $k=1$.

Let $k\geq2$. Assume that \eqref{la} is true for all $1<k^\prime<k$. For $k$, it follows from the definitions of $\ell$ and $\cD$ that
\begin{align*}
(II)_k=&\sum_{\gamma^{(k)}\in\Gamma^{(k)}\backslash\{0\}}\frac{\lozenge^{\ell(\gamma^{(k)})}}{\cD(\gamma^{(k)})}\\
=&\sum_{(\gamma_j^{(k-1)})_{1\leq j\leq P}\in(\Gamma^{(k-1)})^{P}}\frac{\lozenge^{1+\sum_{j=1}^{P}\ell(\gamma_j^{(k-1)})}}{\left(1+\sum_{j=1}^{P}\ell(\gamma_j^{(k-1)})\right)\prod_{j=1}^{P}\cD(\gamma_j^{(k-1)})}\\
\leq&\lozenge\cdot\prod_{j=1}^{P}\sum_{\gamma_j^{(k-1)}\in\Gamma^{(k-1)}}\frac{\lozenge^{\ell(\gamma_j^{(k-1)})}}{\cD(\gamma_j^{(k-1)})}\\
\leq&\left(\frac{P}{2p}\right)^{P}\lozenge.
\end{align*}
Thus
\begin{align*}
\text{the left-hand side of}~~\eqref{la}~~\text{for}~~k~~\text{is}~~&(I)_k+(II)_k
%\leq&1+\left(\frac{P}{2p}\right)^{P}\lozenge\\
\leq1+\left(\frac{P}{2p}\right)^{P}\cdot\frac{(2p)^{2p}}{P^{P}}
%=&1+\frac{1}{2p}\\
=\frac{P}{2p}.
\end{align*}
This shows that \eqref{la} holds for $k $.

By induction, we know that \eqref{la} is true for all $k\in\bN$. This completes the proof of Lemma \ref{gal} .
\end{proof}
\begin{lemm}(Analyticity in space)\label{expp}
Let $f: \mathbb R\rightarrow\mathbb R$ be a quasi-periodic function defined by the Fourier series
\[
f(x)=\sum_{n\in\mathbb Z^\nu}\hat f(n)e^{{\rm i}(n\cdot\omega) x}, \quad x\in\mathbb R,
\]
where $\omega\in\bR^\nu$ is rationally independent.
If the Fourier coefficients $\hat f(n)$ decay exponentially, that is, there exist $\mathtt A>0$ and $0<\rho\leq1$ such that
\[
\hat f(n)\ll\mathtt Ae^{-\rho|n|},\quad\forall n\in\mathbb Z^\nu,
\]
then $f$ is analytic.
\end{lemm}

\begin{proof}
%[Proof of Lemma \ref{expp}]
%It follows from \cite{gevrey} that
It is sufficient to prove that $f$ is a member of the {\bf Gevrey class} of order $1$; see \cite{gevrey}.

First, for any $m\in\mathbb N$, we have
\begin{align*}
(\partial_x^mf)(x)=\sum_{n\in\mathbb Z^\nu}({\rm i}(n\cdot\omega))^m\hat f(n)e^{{\rm i}\langle n\rangle x}.
\end{align*}
It follows from \eqref{ine1}, the exponential decay of $\hat f(n)$, $y^me^{-Ky}\leq(K^{-1})^mm!$ for all $1\leq m\in\mathbb N$, where $K>0$, and \cite[Lemma 2.4(2)]{DLX22AR}
%and  {Lemmas} \ref{l2l}--\ref{l1l}
that
\begin{align*}
(\partial_x^mf)(x)
\ll~&\mathtt A|\omega|^m\sum_{n\in\mathbb Z^\nu}|n|^me^{-\rho|n|}\\
=~&\mathtt A|\omega|^m\sum_{n\in\mathbb Z^\nu}\underbrace{|n|^me^{-\frac{\rho}{2}|n|}}_{\text{bounded by}~(2\rho^{-1})^mm!}e^{-\frac{\rho}{2}|n|}\\
\ll~&\mathtt A(2\rho^{-1}|\omega|)^mm!\sum_{n\in\mathbb Z^\nu}e^{-\frac{\rho}{2}|n|}\\
=~&\mathtt A(2\rho^{-1}|\omega|)^mm!\prod_{j=1}^\nu\sum_{n_j\in\mathbb Z}e^{-\frac{\rho}{2}|n_j|}\\
\ll~&\mathtt A(6\rho^{-1})^\nu(2\rho^{-1}|\omega|)^mm!\\
\ll~&\left(\max\{\mathtt A(6\rho^{-1})^\nu,2\rho^{-1}|\omega|\}\right)^{m+1}m!.
\end{align*}
This implies that $f$ is a member of the Gevrey class of order $1$, that is, $f$ is analytic. This completes the proof of Lemma \ref{expp}.
\end{proof}

\bibliographystyle{alpha}
\bibliography{NLS}

\end{document}